\documentclass[11pt,reqno]{amsart}

\usepackage[a4paper,margin=30mm]{geometry}
\usepackage{amsmath,amssymb,mathtools,mathrsfs}
\usepackage{microtype}
\usepackage{enumitem}
\usepackage{xcolor}
\usepackage{hyperref}
\usepackage[nameinlink,capitalise,noabbrev]{cleveref}

\hypersetup{
	colorlinks=true,
	linkcolor=blue!55!black,
	citecolor=blue!55!black,
	urlcolor=blue!55!black,
	pdftitle={Smoothness of Shortest Isogenies to Frobenius Conjugates}
}

\allowdisplaybreaks
\numberwithin{equation}{section}
\newtheorem{theorem}{Theorem}[section]
\newtheorem*{heurw}{Heuristic W}
\newtheorem*{heurm}{Heuristic M}
\newtheorem{proposition}[theorem]{Proposition}
\newtheorem{lemma}[theorem]{Lemma}
\newtheorem{corollary}[theorem]{Corollary}

\newtheorem{definition}[theorem]{Definition}

\AddToHook{env/proposition/begin}{%
	\crefalias{theorem}{proposition}%
}

\AddToHook{env/lemma/begin}{%
	\crefalias{theorem}{lemma}%
}

\AddToHook{env/corollary/begin}{%
	\crefalias{theorem}{corollary}%
}

\crefname{theorem}{theorem}{theorems}
\Crefname{theorem}{Theorem}{Theorems}

\crefname{proposition}{proposition}{propositions}
\Crefname{proposition}{Proposition}{Propositions}

\crefname{lemma}{lemma}{lemmas}
\Crefname{lemma}{Lemma}{Lemmas}

\crefname{corollary}{corollary}{corollaries}
\Crefname{corollary}{Corollary}{Corollaries}

\DeclareMathOperator{\Tr}{\operatorname{tr}}
\DeclareMathOperator{\trd}{trd}
\DeclareMathOperator{\Aut}{Aut}
\DeclareMathOperator{\End}{End}
\DeclareMathOperator{\Hom}{Hom}
\DeclareMathOperator{\Pic}{Pic}
\DeclareMathOperator{\nrd}{nrd}

\DeclareMathOperator{\Emb}{Emb}
\DeclareMathOperator{\Idl}{Idl}
\DeclareMathOperator{\PIdl}{PIdl}

\newcommand{\F}{\mathbb F}
\newcommand{\Q}{\mathbb Q}
\newcommand{\Z}{\mathbb Z}

\newcommand{\Sp}{\mathcal S_p}
\newcommand{\Mp}{\mathcal M_p}
\newcommand{\Pplus}{P^{+}}
\newcommand{\Prob}{\mathrm{Pr}}
\newcommand{\BH}{\mathrm{BH}}
\newcommand{\EE}{\mathrm E}
\newcommand{\cO}{\mathcal O}
\newcommand{\cA}{\mathcal A}
\newcommand{\cB}{\mathcal B}

\newcommand{\cM}{\mathcal M}

\newcommand{\eps}{\varepsilon}
\newcommand{\ind}[1]{\mathbf 1_{\{#1\}}}
\newcommand{\floor}[1]{\left\lfloor #1\right\rfloor}
\newcommand{\ceil}[1]{\left\lceil #1\right\rceil}
\newcommand{\set}[1]{\left\{#1\right\}}

\title[Shortest smooth Frobenius isogenies]{Smoothness of Shortest Isogenies to Frobenius Conjugates}

\author{Yi Ouyang$^{1,2}$, Wei Wei$^1$ and Zheng Xu$^2$}
\address{$^1$School of Mathematical Sciences,    University of Science and Technology of China, Hefei 230026, Anhui, China}

\address{$^2$Hefei National Laboratory, University of Science and Technology of China, Hefei 230088, China}

\email{yiouyang@ustc.edu.cn}	
\email{veivei@mail.ustc.edu.cn}
\email{xuzheng1@mail.ustc.edu.cn}


\begin{document}
	
	\begin{abstract}
		
		Let \(E\) be chosen uniformly from the
		geometric isomorphism classes of supersingular elliptic curves over
		\(\overline{\mathbb{F}}_p\) and \(\delta_E\) be the minimum degree of isogenies from \(E\) to its Frobenius conjugate \(E^{(p)}\). We establish bounds for the smoothness probability of $\delta_E$
		and hence affirm  Wesolowski's  
		\(p^{1/3+o(1)}\) attack on the
		supersingular endomorphism ring problem is unconditional. We also establish a uniform asymptotic formula, with a main term expressed in terms of Dirichlet \(L\)-values, for the probability that a connection isogeny of degree $n$ exists between $E$ and $E^{(p)}$.
	\end{abstract}
	
	\maketitle
	
	\section{Introduction}
	Isogeny-based cryptography forms a distinctive family of post-quantum constructions grounded in the arithmetic of elliptic curves and quaternion algebras. 
	The landscape of isogeny-based cryptography shifted dramatically after Castryck and Decru, Maino et al., and Robert discovered polynomial-time attacks against SIDH and SIKE \cite{CD23,MMP+23,Rob23}. These attacks leveraged the auxiliary torsion-point information exposed by SIDH and, crucially, demonstrated the power of higher-dimensional isogeny techniques. 
	
	Following the polynomial time cryptanalytic breakthroughs against SIDH\cite{Rob23}, new constructions have branched into two primary research threads. The first seeks to protect sensitive information using new masking strategies or alternative representations, yielding M-SIDH and MD-SIDH \cite{FMP23}, as well as binSIDH and terSIDH constructed from artificial orientations \cite{BF23}. The second reuses higher-dimensional isogeny techniques to build constructive cryptographic primitives. Examples include the trapdoor-based encryption schemes FESTA and QFESTA \cite{BMP23,NO24}, and the unknown-degree isogeny encryption scheme POKÉ \cite{BM25}.
	
	SQIsign, first introduced in 2020, attracted substantial attention due to its compact signatures \cite{DKL+20}. A series of subsequent variants — SQIsignHD \cite{DLRW24}, SQIsign2D-West \cite{BDD+24}, SQIsign2D-East \cite{NOC+24}, SQIsign2D\(^2\) \cite{XLZO25}, and SQIsign2DPush \cite{NO26}—employ higher-dimensional isogeny representations to improve performance and reinforce security guarantees. Collectively, these developments highlight the need for rigorous analysis of the underlying computational problems in supersingular isogeny cryptography. In May 2026, NIST moved SQIsign forward into the third round of its Additional Digital Signatures standardization effort \cite{NISTIR8610}. More recently, Basso et al. \cite{BassoEtAl2026MIKE} introduced MIKE, an efficient and compact post-quantum non-interactive key exchange (NIKE) scheme.
	
	Under their respective security models and auxiliary assumptions, the security of SQIsign and SQIsign variants can be reduced to the hardness of the supersingular endomorphism-ring problem. Wesolowski established polynomial-time reductions between this problem and the supersingular isogeny path-finding problem \cite{Wes22}, allowing cryptanalytic attacks on the underlying hardness assumption to be approached through isogeny path finding. Find a nontrivial endomorphism is important since then.
	
	Prior to the work of Wesolowski \cite{Wesolowski2026}, the best-known classical and quantum algorithms for the general supersingular isogeny path-finding problem had time complexities \(\widetilde O(p^{1/2})\) \cite{DG16} and \(\widetilde O(p^{1/4})\) \cite{BJS14}, respectively.
	
	For a supersingular elliptic curve $E$ in characteristic \(p\), let
	\(E^{(p)}\) denote its Frobenius conjugate, and $\delta_E$ the least
	isogeny degree from \(E\) to \(E^{(p)}\): 
	\[ \delta_E  :=
	\min\limits_{0\neq\varphi\in\Hom(E,E^{(p)})}
	\deg\varphi. \]
	
	\cite{AubryOyonoVincent2026} tells that $\delta_E\leq \lfloor\frac{p}{2}^{\frac{1}{3}}\rfloor$, and \cite{Wesolowski2026} uses this establish a $p^{1/3+o(1)}$ attack. 
	
	However, in isogeny-based cryptography, isogenies of smooth degree are
	particularly amenable to efficient computation, since they can be
	represented and evaluated as compositions of low-degree isogenies
	\cite{DeFeo2017,Wesolowski2026}. For an isogeny of a large prime
	degree~$N$, given a generator of its kernel, the square-root
	V\'elu formulas allow its codomain to be computed and the isogeny to
	be evaluated at a point using
	$\widetilde{O}(\sqrt{N})$ field operations
	\cite{Velu1971,BernsteinDeFeoLerouxSmith2020}.
	This naturally leads to a finer arithmetic question concerning the
	prime factors of~$\delta_E$: how often is $\delta_E$
	$B$-smooth, or equivalently, how is its largest prime factor
	$P^+(\delta_E)$ distributed?

	The recent attack by Wesolowski\cite{Wesolowski2026}   on the
	supersingular endomorphism ring problem underlying the security of
	SQIsign  searches for a \(B\)-smooth isogeny $\varphi \colon E \longrightarrow E^{(p)}$.
	By composing \(\varphi\) with the Frobenius from \(E^{(p)}\)
	back to \(E\), one obtains a non-scalar endomorphism of \(E\).
	Wesolowski gave algorithms for the one-endomorphism and endomorphism-ring
	problems with time and space complexity \(p^{1/3+o(1)}\), under the
	heuristic assumption that the smoothness of \(\delta_E\) behaves like
	that of a uniformly random integer of comparable size. The following heuristic is essential to his attack: 
	
	\begin{heurw}[Heuristic 1, \cite{Wesolowski2026}] Let $p$ be a prime number, and let $E/\F_{p^2}$ be a uniformly random supersingular elliptic curve. Then
		\[\Prob\bigl(P^+(\delta_E)<B\bigr)\approx\rho(u)\approx  u^{-u(1+o(1))},\]
		where $u = \frac{\log(p/2)}{3 \log(B)}$ (uniformly as $p \to\infty$, if $(\log p)^{\eps} < u < (\log p)^{1-\eps}$ for a fixed constant $\eps$) and $\rho$ is the Dickman-de Bruijn function.
	\end{heurw} 
	Mamah~\cite{Mamah2026} stated another heuristic in his analysis of the effectiveness of Wesolowski's algorithm, this heuristic originally commes from \cite{BooherEtAl2024}:
	 \begin{heurm}[Heuristic 1, \cite{Mamah2026}]  For a fixed degree $n$, the probability that a supersingular curve $E$
		 	admits an $n$-isogeny to its Frobenius conjugate $E^{(p)}$ is
		 	\[ \Pr(A_{E,n}) \approx \sqrt{\frac{12n}{p}} .  \]	
		 \end{heurm}

	Our main result in this paper is the following theorem, which implies  the heuristic of Wesolowski as an application: 
	\begin{theorem}\label{thm:main} Fix $\eps\in (0,1/2)$.  There  exist constants $c_1,c_2$ and $M$ depending only on $\eps$, such that
		\begin{equation}\label{eq:main-bounds}
			\frac{c_1\rho(u)}{(\log p)^3}
			\leq
			\Prob\bigl(P^+(\delta_E)<B\bigr)
			\leq
			c_2\rho(u)\log p,
		\end{equation}
		if $p>M$ and 	$\frac{\log (p/2)^\eps}{3}<\log B <\frac{\log (p/2)^{1-\eps}}{3}$ (equivalently $\log (p/2)^\eps< u <\log (p/2)^{1-\eps}$). 
	\end{theorem} 
	We also prove that the predicted logarithmic exponent of the heuristic of Mamah is correct and demonstrate that the proposed universal leading constant does not hold in general.
	
	We briefly mention related studies in this subject. 
	Charles, Goren, and Lauter \cite{CGL09} studied short paths
	between Frobenius-conjugate vertices in supersingular
	$\ell$-isogeny graphs and obtained bounds on the number of
	vertices admitting such paths. 

	Aubry and Vincent related such
	isogenies to inseparable endomorphisms and rank-two sublattices of the
	Gross lattice \cite{AubryVincent2026}, while Aubry, Oyono, and Vincent proved the bound
	\[
	\delta_E
	\le
	\left\lfloor (p/2)^{1/3}\right\rfloor
	\]
	and showed that the exponent $1/3$ is asymptotically optimal
	\cite{AubryOyonoVincent2026}. Complementary recent work studies the extremal
	value $\max_E\delta_E$ \cite{Swa26} and the analogous problem for
	principally polarized superspecial abelian surfaces \cite{Pha26}.
	
	The use of isogenies to a Frobenius conjugate for constructing
	non-scalar endomorphisms appears in
	\cite{EHL+20,FIK+25}.
	Building on the bound of \cite{AubryOyonoVincent2026}, Wesolowski proposed his heuristic, which was then supported by the experiments and
	implementations of \cite{Pan26,AOV26b}. An independent preprint by Delgado gives an unconditional
	$p^{1/3+o(1)}$ algorithm by fixing in advance a family of smooth degrees and showing that isogenies of those degrees occur on
	sufficiently many curves \cite{Del26}.

	The paper is organized as follows. In \Cref{sec:notation} we introduce some necessary background about smooth numbers, class numbers of imaginary quadratic orders and divisor functions. We then study the Brandt operators and primitive connection isogenies about supersingular elliptic curves in \Cref{subsec:SSEC}. \Cref{sec:upper} and \Cref{sec:lower} are devoted to the proof of the  upper and lower bounds respectively via number-theoretic arguments, and \Cref{sec:lower} is the technical heart of this paper.
	We then deduce the heuristic of Wesolowski 
	in \Cref{sec:Impact}. Appendix~\ref{ap:Mhuristic} proves that Mamah’s heuristic holds in the logarithmic sense, giving the existence probability of a connection isogeny between $E$ and $E^{(p)}$ of degree $n$. Appendix~\ref{ap:proofoflemmas} contains detailed proofs of the technical lemmas in \Cref{sec:notation}.
	\subsection*{Acknowledgement} 
	This wok is partially supported by NSFC (Grant No. 12371013 and  12501731) and  QNMP (Grant No. 2021ZD0302902). 
	
	\subsection*{Declaration of AI use} The authors used ChatGPT-5.6 Sol to assist in developing mathematical arguments and derivations presented in this paper. All AI-assisted mathematical results and proofs have been independently verified by the authors, who take full responsibility for the accuracy and integrity of the paper.

	\section{Notation and preliminaries}\label{sec:notation}

	\subsection{Notations and conventions}
	In this work, we assume $p\geq5$ is prime and $k=\overline{\F}_p$. All curves, isogenies, and homomorphisms are geometric unless a field of definition is explicitly specified. We set
	\begin{equation*}
		N=\floor{(p/2)^{1/3}},
		\qquad
		u=\frac{\log (\frac{p}{2})}{3\log B}.
	\end{equation*}

	We let $P^+(1)=1$ and $P^+(n)$ be the largest prime factor of integer $n >1$. Hence $n$ is $B$-smooth is exactly the condition $P^+(n)\leq B$. 
	
	For real functions $f(x)$ and $g(x)$, we denote by \(f=O(g)\) (respectively \(f=O_{\varepsilon}(g)\)) if there exists an absolute constant \(C>0\) (respectively  $C$ depending on $\eps$)  such that \(f \leq Cg\), by \(f\asymp g\) if $f=O(g)$ and $g=O(f)$.

	All probabilities and expectations below refer to the uniform probability measure on the set $\Sp$ of isomorphism classes of supersingular elliptic curves over $k$.
	
	The proofs of the lemmas in this section are deferred to
	Appendix~\ref{ap:proofoflemmas}.
	
	\subsection{Counting smooth numbers}
	\label{subsec:smooth-numbers}
	
	For real numbers \(x\geq 1\)
	and \(y\geq 2\), define
	\[
	\Psi(x,y)
	:=
	\#\left\{
	n\in\Z_{\ge 1}:
	n\leq x,\ P^+(n)\leq y
	\right\}.
	\]
	In other words, \(\Psi(x,y)\) counts \(y\)-smooth positive integers \(\le x\).

	The Dickman-de Bruijn function \(\rho\colon[0,\infty)\to\mathbb{R}\) is the unique continuous function, continuously differentiable on \((1,\infty)\), satisfying the delay differential equation
	
	\begin{equation}\label{eq:dickman-debruijn}
		\begin{cases}
			\rho(v)=1, & 0\leq v\leq 1,\\
			v\rho'(v)+\rho(v-1)=0, & v>1.
		\end{cases}
	\end{equation}
	It is positive and non increasing on \([0,\infty)\).
	
	For every fixed \(v\geq 1\), Dickman's theorem\cite{Dickman1930} (see also
	\cite[(1.1)--(1.3)]{HildebrandTenenbaum1986}.) gives
	\begin{equation}\label{eq:dickman-fixed-v}
		\lim_{x\to\infty}
		\frac{\Psi\left(x,x^{1/v}\right)}{x}
		=
		\rho(v).
	\end{equation}
	Moreover, de Bruijn\cite{deBruijn1951Asymptotic} proved that
	\begin{equation}\label{eq:dickman-asymptotic}
		\lim_{v\to\infty}
		\frac{-\log\rho(v)}{v\log v}
		=1.
	\end{equation}
	Equivalently,
	\(\rho(v)=v^{-v(1+o(1))}\).
	
	\begin{lemma}\label{lem:Psi-uniform}   
		Let $N_0=\left(\frac p2\right)^{1/3}$, $N=\lfloor N_0\rfloor$ and $u=\frac{\log N_0}{\log B}$. Fix $0<\eps<1/2$ and assume $(\log p)^\eps<u<(\log p)^{1-\eps}$. When $p\to \infty$, we have 
		\begin{equation}\label{eq:Psi-rho}
			\frac{\Psi(N,B)}N
			=\rho(u)\left\{1+O_\eps\left(
			\frac{\log(u+1)}{\log B}+\frac1N
			\right)\right\}.
		\end{equation}
		In particular, $\frac{\Psi(N,B)}{N}=\rho(u)(1+o_\eps(1))$. 
	\end{lemma}

	We define the set of squarefree $B$-smooth integers in $[\frac{X}{2},X]$ by
	\begin{equation}\label{eq:shell-def}
		\cA(X,B)=
		\set{n\in\Z_{\geq1}:X/2<n\leq X,\ \mu(n)^2=1,\ \Pplus(n)\leq B},
	\end{equation}
	where $\mu(n)$ is the M\"obius function.
	
	\begin{lemma}\label{lem:squarefree-shell}
		Suppose $(\log X)^3 \le B \le X$ and $\log B \ge (\log\log X)^{5/3+\delta}$ for some fixed $\delta>0$.
		Put $v=\frac{\log X}{\log B}$. If $v\to\infty$, then
		\begin{equation}\label{eq:shell-cardinality}
			\#\cA(X,B)\asymp X\rho(v).
		\end{equation}
	\end{lemma}

	\subsection{Class number formula of imaginary quadratic orders}
	\label{subsec:kronecker-characters}
	
	For a negative quadratic discriminant
	\(D\equiv 0,1\pmod 4\),  let
	\(
	K_D:=\Q(\sqrt{D})\) and \(\cO_D
	:=
	\Z\left[\frac{D+\sqrt{D}}{2}\right]
	\)
	be the imaginary quadratic order of discriminant \(D\). Let
	\[ h(D):=\#\Pic(\cO_D) \]
	be the class number of $\cO_D$  and  
	\[ w(D) :=[\cO_D^\times:\{\pm1\}]
	= \begin{cases}
		3, & D=-3,\\
		2, & D=-4,\\
		1, & D<-4.
	\end{cases} \]

	Any such discriminant $D$ admits a unique factorization
	\begin{equation}\label{eq:discriminant-conductor-decomp}
		D=f^2D_0,
	\end{equation}
	where \(D_0\) is the fundamental discriminant of \(K_D\) and \(f\ge 1\) is the conductor of \(\cO_D\). Equivalently,
	\[
	\cO_D
	=
	\Z+f\cO_{D_0},
	\]
	where \(\cO_{D_0}\) denotes the maximal order of \(K_D\).
	
	We have the following class-number formula, see \cite[Theorem~7.24 and the discussion following Corollary~7.28]{Cox2013} for details.
	\begin{proposition}
		\label{prop:quadratic-order-class-number}
		Let \(D<0\) be an order discriminant. Then
		\begin{equation}\label{eq:quadratic-order-class-number}
			\frac{h(D)}{w(D)}
			=
			\frac{\sqrt{|D|}}{\pi}
			L(1,\chi_D).
		\end{equation}
		where $L(s,\chi_D)$ is the 	Dirichlet \(L\)-function associated to the Kronecker character 
		\[
		\chi_D\colon\Z_{\ge 1}\to\{-1,0,1\}, \ m\mapsto \left(\frac{D}{m}\right), 
		\]
		with $\left(\frac{D}{m}\right)$  the Jacobi symbol modulo $m$. Equivalently, by \(\#\cO_D^\times = 2w(D)\), 
		\begin{equation}L(1,\chi_D)
			= \frac{2\pi h(D)}
			{\#\cO_D^\times\sqrt{|D|}}.
		\end{equation}
	\end{proposition}
	Also, we have the well-known upper bound for the Dirichlet L-function. This is a standard exercise; see \cite{davenport2000multiplicative}.
	\begin{lemma}\label{lem:L-upper}
		If $\chi$ is a nonprincipal Dirichlet character modulo $q$, then
		\[
		|L(1,\chi)|\leq\log q+2.
		\]
	\end{lemma}

	\begin{lemma}\label{lem:tatuzawa}
		Let \(p\) be an odd prime and \(2<X<p/4\). For
		\(n\in\mathcal A(X,B)\), define
		\begin{equation}\label{eq:Dn-def}
			D_n=
			\begin{cases}
				-pn, & pn\equiv3\pmod4,\\
				-4pn, & \text{otherwise}.
			\end{cases}
		\end{equation}
		Then \(D_n\) is the fundamental discriminant of
		\(\mathbb Q(\sqrt{-pn})\), and the map \(n\mapsto D_n\) is injective
		on \(\mathcal A(X,B)\).
		
		For sufficiently large \(p\),  except at most one
		\(n\in\mathcal A(X,B)\), we have
		\begin{equation}\label{eq:tatuzawa-primitive}
			L(1,\chi_{D_n})
			>
			\frac{0.655e^{-2}}{\log p}.
		\end{equation}
		Moreover, if \(pn\equiv3\pmod4\), then the character
		\(\chi_{4D_n}=\chi_{-4pn}\) satisfies
		\begin{equation}\label{eq:tatuzawa-imprimitive}
			L(1,\chi_{-4pn})
			>
			\frac{0.655e^{-2}}{2\log p}.
		\end{equation}
	\end{lemma}
	
\subsection{Divisor-Sum Estimates}
\label{subsec:divisor-sum}

Let \(\tau(k)\) denote the number of positive divisors of
a positive integer \(k\). We record two divisor-sum estimates
that will be used in the subsequent sections.
	
		\begin{lemma}
		\label{lem:uniform-divisor-sum}
		For every integer $M\geq 1$, we have
		\begin{equation}\label{eq:uniform-divisor-sum}
			\sum_{\substack{t\in\mathbb Z\\t^2<4M}}
			\tau(4M-t^2)=O\left(\sqrt{M}\bigl(1+\log(2M)\bigr)^2\right),
		\end{equation}
		where the implied constant is absolute.
	\end{lemma}
	
	\begin{lemma}\label{lem:joint-divisor-sum}
		For every squarefree integer \(n\geq1\),
		\begin{equation}\label{eq:joint-divisor-sum}
			\sum_{m\leq n}
			\sum_{\substack{t\in\Z\\t^2<4mn}}
			\tau(4mn-t^2)
			\leq60n^2\bigl(1+\log(2n)\bigr).
		\end{equation}
	\end{lemma}
	
	
	\section{Counting connection isogenies}\label{subsec:SSEC}
	\subsection{Basic of supersingular elliptic curves}
	Let $k=\overline{\F}_p$ and
	\[
	\Sp=\set{\text{supersingular elliptic curves over }k}/\simeq_k
	\]
	be the set of geometric isomorphism classes. By the usual Deuring class-number count \cite{Deuring1941,Voight2021}, 
	the cardinality $h_p$ of $\Sp$ is given by the formula
	\begin{equation}
		h_p= \left\lfloor\frac{p}{12}\right\rfloor+\epsilon_p, \quad  \text{where}\quad
		\epsilon_p
		=
		\begin{cases}
			0, & p\equiv1\pmod{12},\\
			1, & p\equiv5,7\pmod{12},\\
			2, & p\equiv11\pmod{12}.
		\end{cases}  	
	\end{equation}
	
	For $E\in\Sp$, define
	\[ \delta_E=\min_{0\neq\varphi\in\Hom(E,E^{(p)})}\deg\varphi. \]
	This minimum depends only on the geometric class. By \cite[Theorem~4.2]{AubryOyonoVincent2026},
	\(\delta_E\leq N:=\floor{(p/2)^{1/3}}\) for all $E\in \Sp$.
	
	For \(E\in\Sp\), put
	\[   \cO_E:=\End(E), \quad \End^0(E):=\cO_E\otimes_{\Z}\Q.  \]
	By Deuring's theorem, \(\End^0(E)\) is isomorphic to the quaternion algebra
	\(B_{p,\infty}\) ramified precisely at \(p\) and at the real place, and
	\(\cO_E\) is a maximal order in \(B_{p,\infty}\)
	\cite{Deuring1941}.
	More precisely, after fixing a maximal order
	\(\cO_0\subset B_{p,\infty}\), the geometric supersingular
	classes may be identified with the locally principal left
	\(\cO_0\)-ideal classes, in such a way that the right order
	associated with \(E\) is isomorphic to \(\End(E)\); see
	\cite[Corollary~42.3.7]{Voight2021}.
	
	Let $\pi_E:E\longrightarrow E^{(p)}$ denote the relative Frobenius morphism. It is purely inseparable of
	degree \(p\). Associated with \(\pi_E\) is the two-sided ideal generated by $\pi_E$. 
	\begin{equation}\label{eq:frobenius-ideal}
		\mathfrak P_E
		=
		\left\{
		\beta\circ \pi_E\mid 
		\beta:E^{(p)}\longrightarrow E
		\right\}.
	\end{equation}
	
	\begin{lemma}\label{lem:frobenius-ideal}
		The ideal \(\mathfrak P_E\) is the unique two-sided prime ideal of
		\(\cO_E\) of reduced norm \(p\). It satisfies \(\mathfrak P_E^2=p\cO_E\), \(\trd(\mathfrak P_E)\subseteq p\Z.\)
		Moreover, the map
		\(
		\mathrm{F}_E:
		\Hom(E,E^{(p)})
		\longrightarrow
		\mathfrak P_E\) with \(\varphi\longmapsto\widehat{\varphi}\circ \pi_E\),
		is an isomorphism of additive groups, and \(\nrd\bigl(\mathrm{F}_E(\varphi)\bigr)
		=
		p\deg\varphi\).
		The map \(\mathrm{F}_E\) also preserves scalar primitivity.
	\end{lemma}

	\begin{proof}
		From \cite[\S 42.2.4]{Voight2021}, $ \mathfrak P_E^2=p\mathcal O_E$.
		By \cite[16.6.14]{Voight2021} we have $\overline{\mathfrak{P}_E}=\mathfrak{P}_E$, and \cite[Theorem 18.3.6]{Voight2021} tells $\mathfrak{P}_E\cap\Z=p\Z$, as a consequence, \(\trd(\mathfrak P_E)\subseteq p\Z\).  
		
		By the factorisation property of the relative Frobenius,
		composition with $\pi_E$ induces an isomorphism of
		additive groups
		\[
		\operatorname{Hom}(E^{(p)},E)
		\xrightarrow{} 
		\mathfrak P_E,
		\qquad
		\beta\longmapsto\beta\circ \pi_E.
		\]
		Composing with the duality isomorphism, 
		we obtain the asserted additive isomorphism
		$\mathrm{F}_E(\varphi)=\widehat{\varphi}\circ \pi_E$. For more details about isomorphism, see also \cite[Lemma 42.2.7]{Voight2021}.  Moreover, \(\nrd\bigl(\mathrm{F}_E(\varphi)\bigr)
		=
		\deg\bigl(\widehat{\varphi}\circ \pi_E\bigr)
		=
		\deg\widehat{\varphi}\,\deg \pi_E
		=
		p\deg\varphi.
		\)

	\end{proof}
	
	\subsection{Brandt operators}
	Let
	\[
	\Mp
	:=
	\mathbb Z[\Sp]
	=
	\bigoplus_{[E]\in\Sp}\mathbb Z\,e_{[E]}.
	\]
	For each \(n\geq 1\), the degree-\(n\) isogeny correspondence defines
	a natural \(\mathbb Z\)-linear operator on \(\Mp\), whose matrix records the
	degree-\(n\) isogeny correspondences among the supersingular
	isomorphism classes; see
	\cite[Definition~41.2.4 and Exercise~42.4(a)]{Voight2021}.
	
	\begin{definition}[Brandt operator]
		\label{def:BrandtOperator}
		For \(n\geq 1\), \(T_n \colon \Mp \to \Mp,
		\) is  the \(\mathbb Z\)-linear operator defined by
		\[T_n(e_{[E]})=\sum_{\substack{C\subset E\\ \#C=n}}e_{[E/C]},\]
		After choosing an ordering
		\(\Sp=\{[E_1],\ldots,[E_{h_p}]\},
		\), we denote the matrix of \(T_n\), with respect to the ordered basis
		\((e_{[E_1]},\ldots,e_{[E_{h_p}]})\), by \(T(n)\) in the
		column-source convention, i.e.
		\[
		T(n)_{ij}
		=
		\#\left\{
		C\subset E_j:
		\begin{array}{l}
			C\text{ is a finite subgroup of order }n,\\
			E_j/C\cong E_i
		\end{array}
		\right\}.
		\]
	\end{definition}
	
	When \((n,p)=1\), every subgroup of order $n$ is the kernel of a separable isogeny of degree \(n\). Consequently,
	\begin{equation}\label{eq:brandt-orbits}
		T(n)_{ij}
		=
		\#\left(
		\Aut(E_i)\backslash
		\left\{
		\varphi\colon E_j\to E_i:
		\deg\varphi=n
		\right\}
		\right),
	\end{equation}
	where \(\Aut(E_i)\) acts by postcomposition.
	
	At the ramified prime \(p\), a supersingular elliptic curve $E$ admits a unique isogeny of degree \(p\), namely \(\pi_E\). Hence $T_p(e_{[E]})=e_{[E^{(p)}]}$.
	We write $\mathcal W_p:=T_p$
	for the ramified Brandt operator and $W_p:=T(p)$ for its matrix. Thus \(W_p\) is the permutation matrix induced by $[E]\longmapsto [E^{(p)}]$.
	Since 
	$[E^{(p)}]^{(p)}=[E]$,  $\mathcal W_p^2=\operatorname{id}_{\Mp}$ which corresponds to matrix $I_{h_p}$.
	See \cite[\S42.2.4 and Exercise~42.4(b)--(c)]{Voight2021}.
	
	For \(n\geq 1\), define
	\[
	\mathcal R_E(n)
	:=
	\Aut(E^{(p)})\backslash
	\left\{
	\varphi\colon E\to E^{(p)}\mid 
	\deg\varphi=n
	\right\},
	\]
	where \(\Aut(E^{(p)})\) acts by postcomposition: \(u\cdot\varphi:=u\circ\varphi\). 
	We refer to the elements of \(\mathcal R_E(n)\) as
	\emph{degree-\(n\) connection isogenies}, and set
	\[r_E(n):=\#\mathcal R_E(n). \]
	Thus, \(r_E(n)\) counts orbits of degree-\(n\) connection isogeny from \(E\) to \(E^{(p)}\). With the column-source convention for the Brandt matrices, one has
	\begin{equation}
		r_E(n)=T(n)_{E^{(p)},E}.
	\end{equation}
	Consequently,
	\begin{equation} \label{eq:brandt-orbits2}
		\sum_{E\in\Sp}r_E(n)
		=\sum_{E}\sum_{F} (W_p)_{E,F}(T(n))_{F,E}=
		\operatorname{tr}\bigl(W_pT(n)\bigr).    
	\end{equation}

	There is an equivalent ideal-theoretic description of the Brandt
	operator, obtained through the Deuring correspondence recalled in
	\Cref{subsec:SSEC}; see
	\cite[Corollary~42.3.7]{Voight2021}. Recall the maximal order
	\(\mathcal O_0\subset B_{p,\infty}\) fixed in \Cref{subsec:SSEC}. Under the Deuring correspondence, the
	geometric supersingular classes are identified with the locally
	principal left \(\mathcal O_0\)-ideal classes. Choose representatives
	\(I_1,\ldots,I_{h_p}\)
	indexed compatibly with
	\(\Sp=\{[E_1],\ldots,[E_{h_p}]\}\), so that \(\mathcal O_R(I_i)\simeq\End(E_i)\).
	
	With this indexing, the same Brandt matrix admits the equivalent
	ideal-theoretic description
	\begin{equation}
		\label{eq:brandt-ideal}
		T(n)_{ij}
		=
		\#\left\{
		J\subset I_j:
		\operatorname{nrd}(J)
		=
		n\,\operatorname{nrd}(I_j),
		\quad
		[J]_L=[I_i]_L
		\right\}.
	\end{equation}
	Equation~\eqref{eq:brandt-ideal} is precisely the ideal-class realization of the geometric Brandt operator of
	Definition~\ref{def:BrandtOperator}.
	\footnote{Voight formulates Brandt matrices using right ideals\cite{Voight2021}. The left-ideal
		formulation used here is obtained by ideal conjugation, which changes
		the matrix by transposition and hence does not affect its trace.}

	Through the ideal-class realization above, the trace of the Brandt operator
	can be computed as follows.

	Let \(\widehat{\mathbb Z}
	:=
	\prod\limits_{\ell<\infty}\mathbb Z_\ell\)
	be the profinite completion of $\mathbb Z$, where $\ell$ ranges over
	all finite primes. For an imaginary
	quadratic order \(\mathcal O_D\), put \(\widehat{\cO_D}
	:=
	\cO_D\otimes_{\mathbb Z}\widehat{\mathbb Z}
	=
	\prod\limits_{\ell<\infty}
	\cO_{D,\ell}\), 
	where 
	\(
	\cO_{D,\ell}
	:=
	\cO_D\otimes_{\mathbb Z}\mathbb Z_\ell\). Similarly, \(
	\widehat{\cO}_0
	:=
	\cO_0\otimes_{\mathbb Z}\widehat{\mathbb Z}
	=
	\prod\limits_{\ell<\infty}\cO_{0,\ell}\) with \(
	\cO_{0,\ell}
	:=
	\cO_0\otimes_{\mathbb Z}\mathbb Z_\ell\).
	The group 
	$
	\widehat{\cO}^{\times}_0
	=
	\prod\limits_{\ell<\infty}\cO_{0,\ell}^\times$
	is the unit group of the finite adelic order $\widehat{\cO}_0$. 
	
	\begin{definition} 
		A $\mathbb{Z}_{\ell}$-algebra embedding
		\(\iota_\ell\colon
		\mathcal O_{D,\ell}\hookrightarrow\mathcal O_{0,\ell}\)
		is called \emph{optimal} if
		$\iota_\ell\left(K_D\otimes_{\mathbb Q}\mathbb Q_\ell
		\right)\cap \cO_{0,\ell} =\iota_\ell(\cO_{D,\ell})$.
		Let
		\(
		\operatorname{Emb}^{\mathrm{opt}}
		\left(
		\mathcal O_{D,\ell},
		\mathcal O_{0,\ell}
		\right)\) denote the set of optimal embeddings
		\(
		\iota_\ell:
		\mathcal O_{D,\ell}
		\hookrightarrow
		\mathcal O_{0,\ell}\). The local optimal-embedding number is defined by
		\begin{equation}
			\label{eq:local-optimal-embedding-number}
			m\!\left(
			\mathcal O_{D,\ell},
			\mathcal O_{0,\ell};
			\mathcal O_{0,\ell}^{\times}
			\right)
			:=
			\#
			\left(
			\mathcal O_{0,\ell}^{\times}
			\backslash
			\operatorname{Emb}^{\mathrm{opt}}
			\left(
			\mathcal O_{D,\ell},
			\mathcal O_{0,\ell}
			\right)
			\right),
		\end{equation}
		where \(\mathcal O_{0,\ell}^{\times}\) acts by conjugation on the target:
		\(
		(u\cdot\iota)(x)
		:=
		u\,\iota(x)\,u^{-1}.
		\)
	\end{definition}
	The corresponding finite adelic optimal embedding number is
	\begin{equation}
		\label{eq:adelic-embedding-product}
		m\!\left(
		\widehat{\mathcal O_D},
		\widehat{\mathcal O_0};
		\widehat{\mathcal O_0}^{\times}
		\right)
		=
		\prod_{\ell<\infty}
		m\!\left(
		\mathcal O_{D,\ell},
		\mathcal O_{0,\ell};
		\mathcal O_{0,\ell}^{\times}
		\right).
	\end{equation}
	
	\begin{lemma}\label{lem:Speciale-Eichler-Trace}
		For a nonsquare positive integer $N$, 
		\begin{equation}
			\operatorname{tr}T(N)
			=
			\frac12
			\sum_{\substack{t\in\mathbb Z\\t^2<4N}}
			\;
			\sum_{\substack{D f^2=t^2-4N\\
					D<0}}
			\frac{h(D)}{w(D)}
			m\!\left(
			\widehat{\cO_D},
			\widehat{\cO_0};
			\widehat{\cO_0}^{\times}
			\right).
			\label{eq:eichler-trace-over-Q}
		\end{equation}
	\end{lemma}
	This is the specialization to \(F=\mathbb Q\), \(R=\mathbb Z\), and
	the maximal order \(\mathcal O_0\subset B_{p,\infty}\) of the Eichler
	trace formula; see
	\cite[Theorem~41.5.2]{Voight2021}.

	\subsection{Primitive connection isogenies}
	
	\begin{definition}
		A connection isogeny
		\(\varphi\colon E\to E^{(p)}\) is called \emph{primitive} if  $\varphi\neq [d]_{E^{(p)}}\circ\psi$ for  any $d>1$ and $\psi:E\to E^{(p)}$.	
	\end{definition}
	
	Clearly, by definition, a connection isogeny whose degree is squarefree  or minimal  must be primitive. 
	
	Let $\mathcal R_E^\circ(n)\subseteq\mathcal R_E(n)$
	denote the subset of primitive connection isogenies, and put
	\[ r_E^\circ(n):=\#\mathcal R_E^\circ(n).\]
	Decomposing a connection isogeny by its scalar factor gives
	\begin{equation}\label{eq:content-decomp}
		r_E(n)
		=
		\sum_{d^2\mid n}
		r_E^\circ\left(\frac{n}{d^2}\right).
	\end{equation}
	
	We define the total mass of primitive Brandt–Hurwitz type of degree \(n\) by
	\begin{equation}\label{eq:ABH-def}
		A_p^{\BH}(n)
		:=
		\sum_{E\in\Sp}r_E^\circ(n).
	\end{equation}

	\begin{proposition}\label{prop:primitive-mass} If $1\leq n<p/4$, then
		\begin{equation}
			A_p^{\BH}(n)=\frac{\sqrt{pn}}{\pi}L(1,\chi_{-4pn})
			+\frac{\sqrt{pn}}{2\pi}
			\ind{pn\equiv3\mod4}
			L(1,\chi_{-pn}).
			\label{eq:primitive-mass}
		\end{equation}
	\end{proposition}
	
	\begin{proof}
		By \eqref{eq:brandt-orbits2}, the  number of all degree-$n$ connection isogenies  is 
		\[
		\sum_{E\in\Sp}r_E(n)
		=\sum_{E\in\Sp}T(n)_{E^{(p)},E}
		=\Tr(W_pT(n))   =\Tr(T(pn)), 
		\]
		where the last equality follows from  $W_pT(n)=T(pn)$ \cite[Proposition~41.3.1(b)]{Voight2021}. We now apply Lemma~\ref{lem:Speciale-Eichler-Trace} to  the Brandt operator \(T(pn)\). 
		
		We next constrain the admissible reduced trace parameters. Let \(\cO_0\subset B_{p,\infty}\) denote the maximal orders entering the Brandt trace formula as we used in \Cref{lem:Speciale-Eichler-Trace}. We need to sum up $\frac{h(D)}{w(D)}m(
		\widehat{\mathcal O_D},
		\widehat{\mathcal O_0};
		\widehat{\mathcal O_0}^{\times})$ for all integer triple $(t,D,f)$ satisfying $Df^2=t^2-4pn$ and $t^2<4pn$.
		
		We now show that only the term with \(t=0\) can contribute.
		Fix a triple \((t,D,f)\) appearing in the Eichler trace formula. Since \(Df^2=t^2-4pn\), \(\Z[x]/(x^2-tx+pn)\subseteq\mathcal{O}_D\) with index \(f\).

		Hence, the embedding image \(\alpha\in\mathcal O_0\) of \(x\) satisfies \(\trd(\alpha)=t,\) and \(
		\nrd(\alpha)=pn\). The element \(\alpha\) lies in the
		unique maximal two-sided ideal \(\mathfrak P_p\) of
		\(\mathcal O_0\).  As
		\(\trd(\mathfrak P_p)\subseteq p\Z\), we obtain \(p\mid t\).  On the other hand, the trace formula restricts to
		\(t^2<4pn<p^2\), where the last inequality follows from \(n<p/4\).  Thus
		\(|t|<p\), and consequently
		\(
		t=0\).

		When \(t=0\), the imaginary quadratic order generated by \(\alpha\) has discriminant \(t^2-4pn=-4pn\). \(\mathcal{O}_D\) has discriminant
		\(D=-\frac{4pn}{f^2}\) satisfying
		\(f^2\mid4pn\) and \(D\equiv0,1\pmod4\).
		
		We now evaluate the local optimal-embedding numbers
		\(
		m(
		\mathcal O_{D,\ell},
		\mathcal O_{0,\ell};
		\mathcal O_{0,\ell}^{\times})\). At \(\ell=p\), since \(f^2\mid4pn\),
		\(v_p(D)=1\). It follows that \(\mathcal O_{D,p}\) is the maximal order in a
		ramified quadratic extension of \(\mathbb Q_p\). Moreover,
		\(
		B_{p,\infty}\otimes_{\mathbb Q}\mathbb Q_p\) is the quaternion division algebra over \(\mathbb Q_p\), and
		\(\mathcal O_{0,p}\) is its maximal order. Therefore,
		\cite[Proposition~30.5.3(b)]{Voight2021} gives
		\(
		m\!\left(
		\mathcal O_{D,p},
		\mathcal O_{0,p};
		\mathcal O_{0,p}^{\times}
		\right)
		=1\).

		For every finite prime \(\ell\neq p\), the quaternion algebra splits:
		\(
		B_{p,\infty}\otimes_{\mathbb Q}\mathbb Q_\ell
		\simeq M_2(\mathbb Q_\ell)\). Moreover, \(\mathcal O_{0,\ell}\) is a maximal order and is therefore
		conjugate to \(M_2(\mathbb Z_\ell)\). Hence
		\cite[Proposition~30.5.3(a)]{Voight2021} gives
		\(m\!\left(
		\mathcal O_{D,\ell},
		\mathcal O_{0,\ell};
		\mathcal O_{0,\ell}^{\times}
		\right)
		=1\). Together with the corresponding equality at \(\ell=p\), this shows
		that every finite local optimal-embedding factor in the trace formula
		is equal to \(1\). Consequently,
		\[
		m\!\left(
		\widehat{\mathcal O_D},
		\widehat{\mathcal O_0};
		\widehat{\mathcal O_0}^{\times}
		\right)
		=
		\prod_{\ell<\infty}
		m\!\left(
		\mathcal O_{D,\ell},
		\mathcal O_{0,\ell};
		\mathcal O_{0,\ell}^{\times}
		\right)
		=1.
		\]

		The trace formula therefore becomes
		\begin{equation}\label{eq:total-mass-H}
			\sum_{E\in\Sp}r_E(n)=\operatorname{tr}T(pn)=\frac{1}{2}
			\sum_{\substack{d^2\mid 4pn\\ 4pn/d^2\equiv 0,3\mod 4}} \frac{h(-4pn/d^2)}{w(-4pn/d^2)}. 
		\end{equation}
		Write
		\[   \beta_p(n)=\frac12\biggl(
		\frac{h(-4pn)}{w(-4pn)}
		+\ind{pn\equiv3\mod4}
		\frac{h(-pn)}{w(-pn)}
		\biggr). \]
		Then
		\[ \sum_{E\in\Sp}r_E(n)=\sum_{d^2\mid n}  \beta_p(n/d^2) .\]
		On the other hand, 
		\[ \sum_{E\in\Sp}r_E(n)=\sum_{d^2\mid n} \sum_{E\in \Sp}  r_E^\circ(n/d^2)=\sum_{d^2\mid n} A^{\BH}_p(n/d^2). \]
		Hence 
		\begin{equation}\label{eq:primitive-class-number}
			A_p^{\BH}(n)=\beta_p(n)
			=\frac12\biggl(
			\frac{h(-4pn)}{w(-4pn)}
			+\ind{pn\equiv3\mod4}
			\frac{h(-pn)}{w(-pn)}
			\biggr).
		\end{equation}
		Apply the class number formula in Proposition~\ref{prop:quadratic-order-class-number}  into \eqref{eq:primitive-class-number} then gives \eqref{eq:primitive-mass}.
	\end{proof}
	
	\section{Proof of the upper bound}\label{sec:upper}
	We are ready to prove the upper bound:
	\begin{proposition}\label{prop:upper}
		For every fixed $0<\eps<1/2$, there is a constant $c_2>0$ such that the following holds for all sufficiently large $p$ and all $B$ satisfying $L^\eps<u<L^{1-\eps}$:
		\[
		\Prob\bigl(\Pplus(\delta_E)\leq B\bigr)
		\leq c_2\rho(u)\log p.
		\]
	\end{proposition}
	
	\begin{proof}
		A shortest connection is primitive, and $\delta_E\leq N$. Hence, pointwise in $\Sp$,
		\[
		\ind{\Pplus(\delta_E)\leq B}
		\leq
		\sum_{\substack{1\leq n\leq N\\\Pplus(n)\leq B}}r_E^\circ(n).
		\]
		Averaging over $\Sp$ and using \eqref{eq:ABH-def}, we have:
		\begin{equation}\label{eq:upper-first-moment}
			\Prob\bigl(\Pplus(\delta_E)\leq B\bigr)
			\leq\frac1{h_p}
			\sum_{\substack{n\leq N\\\Pplus(n)\leq B}}A_p^{\BH}(n).
		\end{equation}
		For sufficiently large $p$,
		by \Cref{prop:primitive-mass} and \Cref{lem:L-upper},
		\begin{equation}\label{eq:ap-upper}
			A_p^{\BH}(n)
			\leq\frac{3\sqrt{pn}}{2\pi}\bigl(\log(4pn)+2\bigr).
		\end{equation}
		Consequently, for some absolute constant $C_{upper}>0$: 
		\begin{align}
			\Prob\bigl(\Pplus(\delta_E)\leq B\bigr)
			&\leq
			\frac{3\sqrt{pN}}{2\pi h_p}
			\bigl(\log(4pN)+2\bigr)\Psi(N,B) \notag\\
			&\leq C_{upper}(\log p)\frac{\Psi(N,B)}N,
			\label{eq:upper-Psi}
		\end{align}
		where we used $h_p\asymp p$ and $N\asymp p^{1/3}$. Applying Lemma~\ref{lem:Psi-uniform} proves the result.
	\end{proof}

	\section{Proof of the lower bound}\label{sec:lower}
	

	\subsection{Quaternionic pairs estimate}
	For \(1\leq m,n<p/4\), define the count of ordered distinct pairs
	\begin{equation}\label{eq:Cp-def}
		C_p(m,n)=
		\begin{cases}
			\displaystyle\sum_{E\in\Sp}r_E^\circ(m)r_E^\circ(n),&m\neq n,\\[2mm]
			\displaystyle\sum_{E\in\Sp}r_E^\circ(n)\bigl(r_E^\circ(n)-1\bigr),&m=n.
		\end{cases}
	\end{equation}
	Thus $C_p(m,n)$ counts ordered pairs of distinct primitive connection isogenies of degrees $m$ and $n$ on a common curve. We will count the number of pairs by counting corresponding maximal order. 
	
	For each \(E\in\Sp\) and every integer \(1\leq d<p/4\), we fix, once and for all, a representative \(\varphi_x:E\to E^{(p)}\) for each \(x\in\mathcal R_E^\circ(d)\). These representatives are kept fixed for the entirety of the counting argument.
	
	Fix a supersingular elliptic curve \(E\) and let \(\varphi_m,\varphi_n:E\to E^{(p)}\) represent two distinct primitive connection-isogeny of degrees \(m\) and \(n\), respectively. We associate to this ordered pair a triple \((m,n,t)\) and a corresponding quaternion order \(R\subset B_{p,\infty}\).
	
	Recall the the additive isomorphism established in \Cref{lem:frobenius-ideal}:
	\begin{equation}\label{eq:Theta-def}  \mathrm{F}_E:\Hom(E,E^{(p)})\longrightarrow\mathfrak P_E
	\end{equation}
	where $ \varphi\longmapsto\widehat\varphi\circ \pi_E$. 
	
	We first apply \(\mathrm{F}_E\) to the ordered pair \((\varphi_m,\varphi_n)\), obtaining
	\(
	i:=\mathrm{F}_E(\varphi_m), j:=\mathrm{F}_E(\varphi_n).
	\) 
	Then \(\operatorname{Nrd}(i)=pm\) and \(\operatorname{Nrd}(j)=pn\). 
	Consider the \(\mathbb Z\)-module
	\(
	R:=\mathbb Z[1,i,j,ij]\subseteq B_{p,\infty}\). We will show that \(R\) is a full rank order in $B_{p}$.
	
	We now associate a integer triple \((m,n,t)\) with the isogeny pair \((\varphi_m,\varphi_n)\) by the order $R$
	\footnote{
		Strictly speaking, the parameter $t$ depends on
		the choice of representatives for the primitive
		connection isogeny pairs. For example, replacing
		$\varphi_m$ by $-\varphi_m$ sends $t$ to $-t$. Throughout this counting argument, we fix one representative for each pair once and for all. With these choices, every distinct ordered pair determines a unique parameter $t$,
		and all subsequent upper bounds are independent of the choices.
	}. Throughout our counting argument, we restrict to primitive connection isogenies of degrees \(m,n<X<p/4\).
	Since \(i,j\in\mathfrak P_E\), we have
	\(
	\operatorname{trd}(i),\operatorname{trd}(j)\in p\mathbb Z\). Moreover, $m<X<p/4$, so \(
	\operatorname{trd}(i)^2
	\leq 4\operatorname{Nrd}(i)
	=4pm<p^2\). Thus \(\operatorname{trd}(i)=0\), and similarly \(\operatorname{trd}(j)=0\). 
	
	Furthermore, since \(\mathfrak P_E^2=p\mathcal O_E\), we have \(ij\in p\mathcal O_E\), and hence \(\operatorname{trd}(ij)\in p\mathbb Z\). Using \(\overline{i}=-i\) and \(\overline{j}=-j\), we obtain
	\(
	\operatorname{trd}(ij)
	=ij+\overline{ij}
	=ij+ji\in p\mathbb Z\). 
	Consequently, there exists a unique integer \(t\) such that
	\begin{equation}\label{eq:ij-relations}
		i^2=-pm,\qquad
		j^2=-pn,\qquad
		ij+ji=-pt.
	\end{equation}
	These relations associate the triple \((m,n,t)\) with the ordered pair \((\varphi_m,\varphi_n)\) and determine the multiplication relations in \(R=\mathbb Z[1,i,j,ij]\). Moreover, a direct computation gives $\operatorname{discrd}(R)=p^2(4mn-t^2)$. If \(\operatorname{discrd}(R)=0\), then \(i\) and \(j\) are linearly dependent over \(\mathbb Q\). This implies \(i=\pm j\). By the injectivity of \(\mathrm{F}_E\), we obtain \(\varphi_m=\pm\varphi_n\), contradicting the assumption that they represent distinct connection isogeny classes. Thus \(\operatorname{discrd}(R)>0\). This shows that \(R=\mathbb Z[1,i,j,ij]\)
	is a full rank order in \(B_{p, \infty}\), with reduced discriminant
	\(\operatorname{discrd}(R)=p^2(4mn-t^2)\).

	The bijection \(\mathrm{F}_E\) associates a triple \((m,n,t)\), with \(4mn-t^2>0\), to each pair of distinct primitive connection isogenies. However, counting such triples alone does not suffice to count the number of isogeny pairs, since multiple pairs may give rise to the same triple. We therefore need how many pairs of primitive connection isogenies corresponding to each triple \((m,n,t)\).
	
	For $1\leq m,n<p/4$ and $E\in\Sp$, define
	\[
	\mathscr P_E(m,n):=
	\left\{
	(x,y)\in
	\mathcal R_E^\circ(m)\times
	\mathcal R_E^\circ(n):x\neq y
	\right\}.
	\]
	$C_p(m,n)$ counts all pairs of connecting isogenies of degrees $m$ and $n$ over all curves in $\Sp$, so 
	\(
	C_p(m,n)=
	\sum\limits_{E\in\Sp}\#\mathscr P_E(m,n)\).
	
	For each $(x,y)\in\mathscr P_E(m,n)$, set \(
	t_E(x,y):=
	-\frac{1}{p}\operatorname{trd}
	\left(
	\mathrm{F}_E(x)
	\mathrm{F}_E(y)
	\right)\). By the preceding argument,
	$t_E(x,y)$ is an integer satisfying
	$t_E(x,y)^2<4mn$.
	
	For each integer $t$, set the subset of $\mathscr P_E(m,n)$ consists of pairs with $t_E(x,y)=t$ 
	\[
	\mathscr P_{E,t}(m,n):=
	\left\{
	(x,y)\in
	\mathscr P_E(m,n):\operatorname{trd}\left(\mathrm{F}_E(x)\mathrm{F}_E(y)\right)=-pt
	\right\},
	\]
	and define
	\begin{equation}\label{eq:Cpt-definition}
		C_{p,t}(m,n):=
		\sum_{E\in\Sp}
		\#\mathscr P_{E,t}(m,n).
	\end{equation}
	Since every pair of distinct connection isogenies yields a unique integer $t$,
	we obtain
	\begin{equation}\label{eq:Cp-sum-t}
		C_p(m,n)=\sum_{\substack{t\in\Z\\t^2<4mn}}C_{p,t}(m,n).
	\end{equation}
	
	We next classify the isogeny pairs associated with a fixed triple \((m,n,t)\) according to the isomorphism class of \(\End(E)\). For each maximal-order class \([\mathcal M]\) in \(B_{p,\infty}\), let \(C_{p,t}(m,n;[\mathcal M])\) denote the number of such pairs over all \(E\in\mathcal S_p\) with \(\End(E)\cong\mathcal M\). That is, we define
	\[
	C_{p,t}(m,n;[\mathcal M]):=
	\sum_{\substack{E\in\Sp\\
			\End(E)\simeq\mathcal M}}
	\#\mathscr P_{E,t}(m,n).
	\]\color{black}
	Then
	\[
	C_{p,t}(m,n)
	=
	\sum_{[\mathcal M]}C_{p,t}(m,n;[\mathcal M]).
	\]
	Consequently,
	\[
	C_p(m,n)
	=
	\sum_{\substack{t\in\mathbb Z\\t^2<4mn}}
	\sum_{[\mathcal M]}
	C_{p,t}(m,n;[\mathcal M]).
	\]
	We therefore bound \(C_p(m,n)\) by estimating \(C_{p,t}(m,n;[\mathcal M])\) for each fixed triple \((m,n,t)\) and maximal-order class \([\mathcal M]\).
	
	Fix a triple \((m,n,t)\) with \(D=4mn-t^2>0\), and choose a corresponding pair \((i,j)\) defining the order
	\(
	R=\mathbb Z[1,i,j,ij]\subset B_{p,\infty}\). Since the relations in \eqref{eq:ij-relations} determine the multiplication table on the basis \((1,i,j,ij)\), any other pair \((i',j')\) satisfying the same relations induces a \(\mathbb Q\)-algebra automorphism of \(B_{p,\infty}\) sending \(i\) to \(i'\) and \(j\) to \(j'\). By the Skolem–Noether theorem \cite[Theorem~7.1.3]{Voight2021}, this automorphism is given by conjugation by some \(b\in B_{p,\infty}^\times\). Consequently,
	\[
	R'=\mathbb Z[1,i',j',i'j']=bRb^{-1}.
	\]
	
	Now fix a maximal order \(\mathcal M\subset B_{p,\infty}\). Every pair \((i',j')\) satisfying \eqref{eq:ij-relations} and lying in \(\mathcal M\) determines an embedding
	\(
	\iota:R\hookrightarrow\mathcal M\), where \(\iota(i)=i', \iota(j)=j'\). 
	Thus, for a fixed curve \(E\) with \(\End(E)\cong\mathcal M\), the number of admissible primitive connection-isogeny pairs associated with \((m,n,t)\) is bounded by the number of such embeddings. 
	
	We next establish a bijection between the set of embeddings and a suitable quotient of the set of conjugating elements.
	
	For a maximal order $\cM\subset B_{p,\infty}$, let $\Emb^{\mathrm{lab}}(R,\cM)$ denote the set of unital \(\mathbb Z\)-algebra embeddings
	\(\iota:R\hookrightarrow\mathcal M\), recorded by the ordered pair
	\((\iota(i),\iota(j))\). Put $N_{B_{p,\infty}^{\times}}(\cM)$ the normalizer of $\cM$ in $B_{p,\infty}$. Define
	\[\Omega_R([\cM])
	=\set{\cM'\subset B_{p,\infty}:R\subseteq\cM',\ \cM'\text{ maximal and conjugate to }\cM},\]
	\[A(\cM)=N_{B_{p,\infty}^\times}(\cM)/(\Q^\times\cM^\times),\]
	and counting them by \(N_R([\cM])=\#\Omega_R([\cM])\) and 
	\(a_\cM=\#A(\cM).\)
	
	\begin{lemma}\label{lem:emb-overorders}
		There is a natural identification
		\begin{equation}\label{eq:emb-double-coset}
			\cM^\times\backslash\Emb^{\mathrm{lab}}(R,\cM)
			\simeq
			\cM^\times\backslash
			\set{b\in B_{p,\infty}^\times:bRb^{-1}\subseteq\cM}/\Q^\times.
		\end{equation}
		The map $\Gamma \colon \cM^\times\backslash
		\set{b\in B_{p,\infty}^\times:bRb^{-1}\subseteq\cM}/\Q^\times\to \Omega_R([\cM])$ with \([b]\longmapsto b^{-1}\cM b\)
		is surjective, and every fiber is an $A(\cM)$-torsor. Consequently,
		\begin{equation}\label{eq:emb-overorder-count}
			\#\bigl(\cM^\times\backslash\Emb^{\mathrm{lab}}(R,\cM)\bigr)
			=a_\cM N_R([\cM]).
		\end{equation}
	\end{lemma}

	\begin{proof}
		Since $R$ is full rank so every embedding has the form $r \mapsto brb^{-1}$ (for extending to $B_{p,\infty}$ is uniquely up to automorphism of $B_{p,\infty}$). Two conjugates determine the same embedding exactly when they differ on the right by $\mathbb{Q}^\times$. So the first statement is true.
		
		The condition $bRb^{-1} \subseteq \cM$ is equivalent to $R \subseteq b^{-1}\cM b$, so $\Gamma$ is well defined and surjective.
		
		The normalizer $N(\cM)$ acts on $\cM^\times \backslash \bigl\{ b \in B_{p,\infty}^\times : bRb^{-1} \subseteq \cM \bigr\} / \mathbb{Q}^\times$ by left multiplication: $g\cdot [b] := [gb]$. This action is well defined and preserves image of $\Gamma$. The subgroup $\mathbb{Q}^\times \cM^\times$ acts trivially, therefore the action descends to $A(\cM) = N(\cM)/(\mathbb{Q}^\times\cdot \cM^\times)$.
		
		If $\Gamma([b_1]) = \Gamma([b_2])$, then $b_1^{-1}\cM b_1 = b_2^{-1}\cM b_2$, which means $b_1b_2^{-1} \in N(\cM)$ and $b_1A(\cM) = b_2A(\cM)$. Therefore, the kernel of $\Gamma$ is exactly $A(\cM)$. In particular, we have
		\[
		\#\left(\cM^\times \backslash \mathrm{Emb}^{\mathrm{lab}}(R,\cM)\right) = a_M N_R([\cM]).
		\]
	\end{proof}

	For a class of maximal order $[\cM]$, consider the number of supersingular curve classes with endomorphism ring in $[\cM]$:
	\[
	c_\cM=\#\set{E\in\Sp:\End(E)\simeq\cM}.
	\]
	
	\begin{lemma}\label{lem:type-fiber}
		For every maximal order $\cM\subset B_{p,\infty}$,
		\begin{equation}\label{eq:type-fiber-product}
			a_\cM c_\cM=2.
		\end{equation}
		Consequently,
		\begin{equation}\label{eq:type-fiber-emb}
			c_\cM\,
			\#\bigl(\cM^\times\backslash\Emb^{\mathrm{lab}}(R,\cM)\bigr)
			=2N_R([\cM]).
		\end{equation}
	\end{lemma}
	
	\begin{proof}
		Consider type map:\[
		\mathbf{type}\colon \mathbf{Cl}(\cO_0)\to \{[\cM]\mid \cM \subset B_{p,\infty} \text{ and maximal}\}
		\]
		with $\mathbf{type}([I])=[\cO_R(I)]$.
		The fiber of the type map over $[\cM]$ is identified with
		\(
		\Idl(\cM)/\PIdl(\cM)\),
		the group of invertible two-sided fractional ideals modulo principal two-sided ideals; see \cite[Proposition~18.5.10, Corollary~42.3.7, and Lemma~42.4.1]{Voight2021}. Hence
		\[
		c_\cM=\#\bigl(\Idl(\cM)/\PIdl(\cM)\bigr).
		\]
		Here $\Pic(\cM)$ denotes the group of isomorphism classes of invertible $\Z$-central $\cM$-bimodules. We then have the exact sequence
		\begin{equation}\label{eq:normalizer-exact}
			1\longrightarrow A(\cM)
			\longrightarrow\Pic(\cM)
			\longrightarrow\Idl(\cM)/\PIdl(\cM)
			\longrightarrow1;
		\end{equation}
		see \cite[Theorem~18.1.3 and Proposition~18.5.3]{Voight2021}. Since $\Z$ has trivial ideal class group and $p$ is the only finite ramified prime of $B_{p,\infty}$, the group $\Pic(\cM)$ has order two, generated by the two-sided prime above $p$. Taking orders in \eqref{eq:normalizer-exact} proves \eqref{eq:type-fiber-product}; \eqref{eq:type-fiber-emb} follows from \Cref{lem:emb-overorders}.
	\end{proof}
	
	Combining the preceding results, we obtain a upper bound for \(C_{p,t}(m,n;[\mathcal M])\).
	
	\begin{lemma}\label{lem:pairs-embeddings}
		Let $(m,n,t)$ be a signed triple arising from an isogeny pair $(\varphi_m,\varphi_n)$ with \eqref{eq:ij-relations}, and let $R=\Z[1,i,j,ij]$ be the corresponding order. Then
		\begin{equation}\label{eq:pairs-emb-type}
			C_{p,t}(m,n;[\cM])
			\leq c_\cM\#\cM^\times
			\#\bigl(\cM^\times\backslash\Emb^{\mathrm{lab}}(R,\cM)\bigr).
		\end{equation}
		Let \(N_{\max}(R)
		=\#\set{\cM'\subset B_{p,\infty}:R\subseteq\cM',\ \cM'\text{ maximal}}\), then
		\begin{equation}\label{eq:pairs-Nmax}
			C_{p,t}(m,n)\leq48N_{\max}(R).
		\end{equation}
	\end{lemma}
	
	\begin{proof}
		Consider supersingular elliptic curve $E$, $\iota_E\colon \End(E)\cong \cM$. The isogeny pair can be represented by: 
		\[
		([\varphi_m],[\varphi_n])
		\longmapsto
		\bigl(\iota_E(\mathrm{F}_E(\varphi_m)),\iota_E(\mathrm{F}_E\left(\varphi_n)\right)\bigr).
		\]
		The image satisfies the \eqref{eq:ij-relations} and therefore gives an embedding $R\hookrightarrow\cM$. 
		This induces a map
		\[\Phi_E\colon \mathscr P_{E,t}(m,n)\to \operatorname{Emb}^{\mathrm{lab}}
		(R,\mathcal M)
		\]
		For fixed $E$, $\Phi_E$ is injective. 
		An $\cM^\times$-orbit contains at most $\#\cM^\times$ embeddings, proving \eqref{eq:pairs-emb-type} after summing over the $c_\cM$ curves of this type.
		
		Since $\cM^\times\simeq\Aut(E)$, one has $\#\cM^\times\leq24$ \cite[Theorem~III.10.1]{Silverman2009}. By \Cref{lem:type-fiber},
		\begin{align*}
			C_{p,t}(m,n)
			&\leq24\sum_{[\cM]}
			c_\cM a_\cM N_R([\cM])\\
			&=48\sum_{[\cM]}N_R([\cM])
			=48N_{\max}(R),
		\end{align*}
		where the last equality partitions the maximal orders containing $R$ by type.
	\end{proof}
	
	

	We next derive an upper bound for \(N_{\max}(R)\). 
	For each finite
	prime \(q\), we first bound the number of maximal
	\(\mathbb Z_q\)-orders containing
	\(
	R_q:=R\otimes_{\mathbb Z}\mathbb Z_q,
	\)
	and then combine these local bounds to obtain a global estimate.

	\subsection{Counting maximal orders containing $R$}\label{sec:CounntingMaximalOrders}
	We next derive an upper bound for \(N_{\max}(R)\). Set $D=4mn-t^2$. For each finite prime $q$, put
	\begin{equation}\label{eq:Bq-def}
		\cB_q(R)=
		\set{\cM_q\subset(B_{p,\infty})_q:\cM_q\text{ maximal over }\Z_q,\ R_q\subseteq\cM_q}.
	\end{equation} First, we factor $N_{\max}(R)$ to product of $\cB_q(R)$. 
	Define the restricted product of the sets
	\(\mathcal B_q(R)\) by
	\[
	\prod_q'\mathcal B_q(R)
	=
	\left\{
	(\mathcal M_q)_q\in\prod_q\mathcal B_q(R)
	\mid 
	\mathcal M_q=R_q
	\text{ for all but finitely many }q
	\right\}.
	\] 
	Since
	\(\operatorname{discrd}(R)=p^2D\), the local order \(R_q\) is maximal for \(q\nmid p^2D\), so
	\(\mathcal B_q(R)=\{R_q\}\) for all but finitely many \(q\).  From the local-global
	correspondence, we have the following lemma without proof. 
	\begin{lemma}\label{lem:local-global-overorders} 
		There is an isomorphism: 
		\[\begin{aligned}
			\left\{
			\mathcal M\subset B_{p,\infty}:
			\mathcal M\text{ is a maximal order and }R\subseteq\mathcal M
			\right\} 	& \xrightarrow{\;\sim\;}
			\prod_q'
			\mathcal B_q(R) \\
			\mathcal M	& \longmapsto
			\bigl(\mathcal M\otimes_{\mathbb Z}\mathbb Z_q\bigr)_q. \\
		\end{aligned}\]
		Consequently, we have $N_{\max}(R)
		=
		\prod_q \#\mathcal B_q(R)$.    
	\end{lemma}
	
	So We can get the global bound from the local case. Following computing the local case. Now we will compute the bound of $\# \cB_q(R)$.

	At $q=p$, the division algebra $(B_{p,\infty})_p$ has a unique maximal order, so $\#\cB_p(R)=1$ \cite[Proposition~13.3.4 and Theorem~13.3.11(b)]{Voight2021}.
	
	At \(q\neq p\). Since \(B_{p,\infty}\) is ramified only at \(p\) and \(\infty\), one has
	\(
	B_q\simeq M_2(\Q_q)\). 
	The maximal \(\Z_q\)-orders in \(B_q\) are naturally identified with
	the vertices of the \((q+1)\)-regular Bruhat--Tits tree of
	\(\mathbf{PGL}_2(\Q_q)\); see
	\cite[Definition~23.5.5 and Proposition~23.5.8]{Voight2021}.
	Under this identification, the set
	\(\cB_q(R)\) determines a convex subtree.  
	Set
	\(e_q=v_q(D),
	g_q=\min\{v_q(m),v_q(n)\}\).
	
	If two vertices in $\cB_q(R)$ are at distance $d$, their intersection is an Eichler order of level $q^d$. Since $R_q$ is contained in this intersection, then $d\leq v_q(\operatorname{discrd}(R))=e_q$.     
	Thus $\operatorname{diam}\cB_q(R)\leq e_q$.
	
	Now we only consider the case $n$ or $m$ is square-free. Suppose $m$ is square-free, let $A_i = \mathbb{Z}_q[i]$, $K_i = \mathbb{Q}_q[i]$.
	$\mathcal{O}_{K_i}$ be the maximal $\mathbb{Z}_q$-order. Then we have $A_i = \mathbb{Z}_q + q^{f_i}\mathcal{O}_{K_i}$, for some $f_i \ge 0$. Compute the discriminant on both sides, we have $(-4pm)\mathbb{Z}_q = \operatorname{disc}(A_i) = q^{2f_i}\operatorname{disc}(\mathcal{O}_{K_i})$. Hence $2f_i \le v_q(4pm)$, and then
	\[
	f_i =
	\begin{cases}
		0, & q \neq 2,p,\\
		0,1, & q = 2.
	\end{cases}
	\]
	
	Since $A_i \subseteq R_q$, we have $\cB_q(R) \subseteq \cB_q(A_i)$. Let $S_i = \cB_q(\mathcal{O}_{K_i})$, the standard branch description of quadratic orders
	\cite[\S1]{ArenasCarmonaSaavedra2016}
	gives $\cB_q(A_i) \subseteq N_{f_i}(S_i)$.
	
	We now translate these geometric descriptions into a uniform vertex bound. Let \(T\) be a subtree of the \((q+1)\)-regular Bruhat--Tits
	tree satisfying
	\(\operatorname{diam}(T)\leq e\), \(T\subseteq N_f(S)\), where \(S\) is one of the stems described above.
	Let $T$ be a subtree of the $(q+1)$-regular Bruhat--Tits tree satisfying $\operatorname{diam}(T) \le e$, $T \subseteq N_f(S)$.
	
	If $q \neq 2,p$, $A_i = \mathcal{O}_{K_i}$. Since $S$ is convex, the nearest-point projection $\pi_S : T \to S$ does not increase distances. Thus $\#\pi_S(T) \le e_q + 1$. Moreover, each projection fiber is contained in a ball of radius $f$, which has at most $1$ vertices. Therefore, $\#T \le e_q + 1$. If $q = 2$, $A_i = \mathcal{O}_{K_i}$ or $\mathbb{Z}_q + 2\mathcal{O}_{K_i}$, then each projection fiber at most $1 + (q+1)\displaystyle\sum_{r=0}^{f-1} q^r \le 3q^f$ vertices, so $\#T \le 3(e_2 + 1)q^f \le 6(e_2 + 1)$.

	\begin{theorem}\label{thm:squarefree-pair-bound}
		Let \(1\leq m,n<p/4\), and suppose that \(\mu(n)^2=1\).  For each \(t\) arising from an ordered distinct pair, let \(D=4mn-t^2>0\). 
		Then
		\begin{equation}\label{eq:Cpt-squarefree}
			C_{p,t}(m,n)\leq288\,\tau(D),
		\end{equation}
		where $\tau$ is the divisor function. 
		Consequently,
		\begin{equation}\label{eq:Cp-squarefree}
			C_p(m,n)
			\leq288
			\sum_{\substack{t\in\Z\\t^2<4mn}}
			\tau(4mn-t^2).
		\end{equation}
	\end{theorem}
	\begin{proof}
		
		
		
		From above statement, for a fixed triple $(m, n, t)$, we have
		\[\#\mathcal B_q(R)\leq e_q+1 \ \
		(q\neq2,p);\qquad \# \cB_2(R)\leq 6(e_2+1).\]
		
		Using the local-global product formula \Cref{lem:local-global-overorders}, we obtain
		\begin{align}
			N_{\max}(R)
			&=\prod_q\#\mathcal B_q(R)\notag\\
			&\leq
			6(e_2+1)
			\prod_{\substack{q\mid D\\ q\neq2,p}}
			(e_q+1)
			\leq
			6\tau(D).
			\label{eq:Nmax-squarefree}
		\end{align}
		Here, if \(p\mid D\), the factor \(e_p+1\) occurring in
		\(\tau(D)\) only enlarges the right-hand side, since
		\(\#\mathcal B_p(R)=1\).
		
		Finally, combining \eqref{eq:Nmax-squarefree} with
		\eqref{eq:pairs-Nmax} gives
		\[
		C_{p,t}(m,n)
		\leq
		48N_{\max}(R)
		\leq
		288\tau(D),
		\]
		which proves \eqref{eq:Cpt-squarefree}. Summing over
		\(t\in\mathbb Z\) with \(t^2<4mn\) and using
		\eqref{eq:Cp-sum-t} yields
		\[
		C_p(m,n)
		\leq
		288
		\sum_{\substack{t\in\mathbb Z\\t^2<4mn}}
		\tau(4mn-t^2).
		\]
	\end{proof}

%
	
	\begin{corollary}\label{cor:squarefree-fixed-endpoint}
		Let \(1\leq n<p/4\) be squarefree.  Then
		\begin{equation}\label{eq:squarefree-fixed-endpoint}
			\sum_{m\leq n}C_p(m,n)
			\leq17280\,n^2\bigl(1+\log(2n)\bigr)
		\end{equation}
		In particular, \(\sum_{m\leq n}C_p(m,n)
		=O\left(n^2\log(2n)\right)\).
	\end{corollary}
	
	\begin{proof}
		Combine \Cref{thm:squarefree-pair-bound} with \Cref{lem:joint-divisor-sum}.
	\end{proof}

	\subsection{Bound of $\EE(S_E)$ and $\EE(P_E)$}
	For \(E\in\Sp\), consider the number of all selected-isogeny:
	\begin{equation}\label{eq:SE-def}
		S_E=\sum_{n\in\cA(X,B)}r_E^\circ(n).
	\end{equation}
	On each curve \(E\), order all primitive connection isogenies of degree at most \(X\) by: $\varphi_1<_{\text{order}}\varphi_2$ if $\deg (\varphi_1)<\deg(\varphi_2)$ or $\deg(\varphi_1)=\deg(\varphi_2)$ and $\varphi_1<_{\text{dic}}\varphi_2 $\footnote{Since only an arbitrary total order is required, fix \(E\in\mathcal S_p\) and an integer $n$. The set of primitive connection-isogeny of degree $n$ is finite, so we may list its elements as \([\varphi_{n,1}],\dots,[\varphi_{n,N_{E,n}}]\). We then declare
		\([\varphi_{n,i}] <_{\text{dic}} [\varphi_{n,j}] \iff i<j.\)
		This ordering is arbitrary, but we hold it fixed for the rest of the argument.}. Let \(P_E\) be the number of ordered pairs \((w,z)\) such that
	\begin{itemize}[leftmargin=2em]
		\item \(z\) is selected, meaning \(\deg z\in\cA(X,B)\);
		\item \(w\neq z\) is a primitive connection isogeny of degree at most \(X\);
		\item \(w<_{\text{order}}z\).
	\end{itemize}
	Equivalently,
	\begin{equation}\label{eq:PE-expansion}
		P_E=
		\sum_{n\in\cA(X,B)}
		\left(
		r_E^\circ(n)\sum_{m<n}r_E^\circ(m)
		+\binom{r_E^\circ(n)}2
		\right).
	\end{equation}

	Let \(0<\gamma\leq1\) be a fixed constant and set $\lambda=\gamma(\log p)^{-2}$, $X= \floor{N  \gamma^{2/3} (\log p)^{-4/3}}$ and $v=\log X/\log B$.
	
	\begin{proposition}\label{prop:first-moment} For every \(0<\eps<1/2\), there exists a constant \(a>0\) such that for all sufficiently large \(p\), uniformly for \(B\) satisfying $(\log \frac{p}{2})^{\eps}<u<(\log \frac{p}{2})^{1-\eps}$,
		\begin{equation}\label{eq:SE-lower}
			\EE S_E\geq \frac{a \lambda\rho(v)}{\log p}= \frac{a \gamma \rho(v)}{\log^3 p}\geq \frac{a \lambda\rho(u)}{\log p}= \frac{a \gamma \rho(u)}{\log^3 p}.
		\end{equation}
	\end{proposition}
	
	\begin{proof}
		By \Cref{prop:primitive-mass} and \Cref{lem:tatuzawa}, if $pn\equiv 3\mod 4$, 
		\[  A_p^{\BH}(n)=\frac{\sqrt{pn}}{\pi}L(1,\chi_{-4pn})+\frac{\sqrt{pn}}{2\pi}L(1,\chi_{-pn})> \frac{0.655e^{-2}\sqrt{pn}}{\pi\log p}; \]
		Otherwise
		\[ A_p^{\BH}(n)=\frac{\sqrt{pn}}{\pi}L(1,\chi_{-4pn})
		>\frac{0.655e^{-2}}{\pi\log p}\sqrt{pn}. \]
		As a consequence, outside at most one \(n\in\cA(X,B)\),
		\[
		A_p^{\BH}(n)
		>\frac{0.655e^{-2}}{\pi\log p}\sqrt{pn}.
		\]
		
		Since
		\(
		\log X=\frac{\log \frac{p}{2}}{3}+O(\log\log p)\), \(\log B=\frac{\log \frac{p}{2}}{3u}\), we have
		\(v:=\frac{\log X}{\log B}
		=u(1+o_\varepsilon(1))\to\infty\)
		uniformly for \((\log \frac{p}{2})^\varepsilon<u<(\log \frac{p}{2})^{1-\varepsilon}\) as $p\to \infty$. Moreover,
		\(\log B\geq\frac{(\log \frac{p}{2})^\varepsilon}{3}\), \(\log\log X=\log \log \frac{p}{2}+O(1)\). Hence, for sufficiently large \(p\),
		\(
		(\log X)^3\leq B\leq X\) and, for any fixed \(\delta>0\),
		\(
		\log B\geq(\log\log X)^{5/3+\delta}\). Thus \Cref{lem:squarefree-shell} implies that \(\#\mathcal A(X,B)\asymp X\rho(v)\).
		
		The deletion of one element is negligible because
		\(
		\log(X\rho(v))
		=\frac{\log \frac{p}{2}}{3}-o_\eps(\log \frac{p}{2})\longrightarrow\infty\), using \(v\leq (\log \frac{p}{2})^{1-\eps}(1+o(1))\). We count degrees in $\cA(X,B)$, so it satisfies \(n>X/2\). Hence there exits a constant $C_{lower}$
		\begin{align*}
			\EE S_E
			=\frac1{h_p}\sum_{n\in\cA(X,B)}A_p^{\BH}(n)
			\geq C_{lower}\frac{\sqrt{pX}}{p\log p}\,X\rho(v)
			=C_{lower}\frac{X^{3/2}}{\sqrt p}\frac{\rho(v)}{\log p}.
		\end{align*}
		
		Since $\rho$ is nonincreasing, the inequality
		$v\leq u$ also gives
		\begin{equation}\label{eq:ESE-boundv}
			\EE S_E
			\geq C_{lower}
			\frac{X^{3/2}}{\sqrt p}\frac{\rho(v)}{\log p}
			\geq
			C_{lower}\frac{X^{3/2}}{\sqrt p}\frac{\rho(u)}{\log p}.
		\end{equation}
		Since \(N=N_0+O(1)\) and \(X=N\lambda^{2/3}+O(1)\),
		\begin{equation}\label{eq:3quo2X}
			X^{3/2}
			=N_0^{3/2}\lambda(1+o_\eps(1))
			=\frac{\sqrt p}{\sqrt2}\lambda(1+o_\eps(1)),
		\end{equation}
		Applying \eqref{eq:3quo2X} to \eqref{eq:ESE-boundv} proves that
		\[ \EE S_E\geq \frac{a \lambda\rho(v)}{\log p}=\frac{a \gamma \rho(v)}{\log^3 p}\geq a\frac{\lambda\rho(u)}{\log p} = \frac{a \gamma \rho(u)}{\log^3 p}. \qedhere \]
	\end{proof}

	\begin{proposition}\label{prop:collision} Let \(0<\gamma\leq1\) be a fixed constant and set $X= \floor{N  \gamma^{2/3} (\log p)^{-4/3}}$. Let $v=\frac{\log X}{\log B}$. 
		For every \(0<\eps<1/2\), there is a constant \(b>0\) such that for all sufficiently large \(p\), uniformly for \(B\) satisfying $(\log \frac{p}{2})^{\eps}<u<(\log \frac{p}{2})^{1-\eps}$, 
		\begin{equation}\label{eq:PE-upper}
			\EE P_E
			\leq b\gamma^2\rho(v)\log^{-3} p.
		\end{equation}
		Consequently,
		\begin{equation}\label{eq:collision-ratio}
			\frac{\EE P_E}{\EE S_E}
			\leq  \frac ba\gamma.
		\end{equation}
	\end{proposition}
	
	\begin{proof}
		Summing \eqref{eq:PE-expansion} over \(E\) and using \eqref{eq:Cp-def},
		\begin{align*}
			\sum_{E\in\Sp}P_E
			&=\sum_{n\in\cA(X,B)}
			\left(
			\sum_{m<n}C_p(m,n)+\frac12C_p(n,n)
			\right)\\
			&\leq\sum_{n\in\cA(X,B)}\sum_{m\leq n}C_p(m,n).
		\end{align*}
		Every $n\in\cA(X,B)$ is squarefree, so
		\Cref{cor:squarefree-fixed-endpoint} gives
		\begin{align*}
			\sum_{E\in\Sp}P_E
			&\leq
			17280\sum_{n\in\cA(X,B)}
			n^2\bigl(1+\log(2n)\bigr)\\
			&\leq
			17280 X^2\bigl(1+\log(2X)\bigr)
			\#\cA(X,B)\\
			&\leq C_{P_E} X^3\rho(v)\log p,
		\end{align*}
		where $C_{P_E}>0$ is an absolute constant. 
		The last line uses \eqref{eq:shell-cardinality} and \(X<p\).  Since \(h_p\asymp p\) and \(X^3/p\asymp\lambda^2\), averaging proves \eqref{eq:PE-upper}.  Dividing by \eqref{eq:SE-lower} proves \eqref{eq:collision-ratio}.
	\end{proof}
	
	\subsection{Proof of the lower bound}
	We derive a lower bound for
	\(\Pr(P^+(\delta_E)\leq B)\)
	from estimates for \(\EE[S_E]\) and \(\EE[P_E]\).
	\begin{lemma}\label{lem:first-hit}
		For every \(E\in\Sp\),
		\begin{equation}\label{eq:first-hit}
			\mathbf{1}_{P^+(\delta)\leq B}\geq S_E-P_E.
		\end{equation}
	\end{lemma}
	
	\begin{proof}
		If there is no primitive root of degree at most \(X\), then \(S_E=P_E=0\).  Otherwise, let \(w_{\min}\) be the first root in the chosen ordering.  Its degree is \(\delta_E\), since a shortest connection is primitive.
		
		If \(w_{\min}\) is not selected, every selected root has \(w_{\min}\) as a distinct predecessor.  Hence \(P_E\geq S_E\), and therefore \(S_E-P_E\leq0\leq \mathbf{1}_{P^+(\delta_E)\leq B}\).
		
		If \(w_{\min}\) is selected, then \(\mathbf{1}_{P^+(\delta_E)}=1\). 
		Aside from the shortest isogeny \(w_{min}\), every other selected connection isogeny can be paired with \(w_{min}\) in $P_E$.
		Hence \(P_E\geq S_E-1\), which again gives \(\mathbf{1}_{P^+(\delta_E)}\geq S_E-P_E\).
	\end{proof}

	\begin{proposition}\label{prop:lower}
		For every fixed \(0<\eps<1/2\), there is an absolute constant \(c_1>0\) such that, for all sufficiently large \(p\) and all \(B\) satisfying
		\[
		(\log \frac{p}{2})^\eps<u<(\log \frac{p}{2})^{1-\eps},
		\]
		one has
		\begin{equation}\label{eq:lower-final}
			\Prob\bigl(\Pplus(\delta_E)\leq B\bigr)
			\geq\frac{c_1\rho(u)}{(\log p)^3}.
		\end{equation}
	\end{proposition}
	
	\begin{proof}
		Let \(\gamma=\min\left\{1,\frac{a}{2b}\right\}\), where \(a\) and \(b\) are the constants in \Cref{prop:first-moment,prop:collision}.  Then \eqref{eq:collision-ratio} gives for sufficiently large $p$,
		\[
		\EE P_E\leq\frac12\EE S_E.
		\]
		Taking expectations in \eqref{eq:first-hit} and applying \eqref{eq:SE-lower},
		\begin{align*}
			\Prob\bigl(\Pplus(\delta_E)\leq B\bigr)
			&=\EE [\mathbf{1}_{P^+{\delta_E}\leq B}]
			\geq\EE S_E-\EE P_E
			\geq\frac12\EE S_E
			\geq\frac{a\gamma}{2}\frac{\rho(v)}{(\log p)^3}.
		\end{align*}
		Because \(v\leq u\) and \(\rho\) is decreasing, \(\rho(v)\geq\rho(u)\).  Taking
		\(c_1=\frac{a\gamma}{2}\) proves \eqref{eq:lower-final}.
	\end{proof}

	
	\section{Probabilistic Consequences and Cryptanalytic Applications}\label{sec:Impact}
	
	\subsection{Proof of Wesolowski's heuristic}
	\begin{corollary}\label{cor:log-law} Heuristic W is true. 
	\end{corollary}
	
	\begin{proof}
		Taking logarithms in \eqref{eq:main-bounds} gives
		\begin{equation}\label{eq:log-rho-error}
			\log \rho(u)+\log c_1-3\log\log p\leq\log \Prob\bigl(P^+(\delta_E)<B\bigr)
			\leq \log\rho(u)+\log c_2+\log \log p.
		\end{equation}
		Since $u>  \log (p/2)^{\eps}$, we have: 
		\[\frac{\log\log p}{u\log u}
		<\frac{\log\log p}{\log (p/2)^{\eps}}\longrightarrow 0\text{ as }p \to \infty.\]
		So for any constant $C$ there exists $p_0$, for all $p>p_0$
		\[
		\frac{\log\log p}{u\log u}<C.
		\]
		Combining this with \eqref{eq:dickman-asymptotic} proves what we want. 
	\end{proof}

	By \Cref{cor:log-law}, Wesolowski's heuristic follows from our result. Consequently, the \(p^{1/3+o(1)}\) attack no longer relies on this heuristic assumption and admits an unconditional justification.

	\subsection{A higher success probability}
	\label{subsec:Higher-success-probability}
	In the preceding analysis, \(P_E\) was used to restrict attention to connecting isogenies of minimal degree \(\delta_E\).
	The \(p^{1/3}\)-attack of~\cite{Wesolowski2026}, however, does not
	require the connecting isogeny recovered by the attack to have minimal
	degree. The attack first enumerates all isogenies emanating from \(E\)
	whose degrees are \(B\)-smooth and at most \(B^{1/2}\left(\frac{p}{2}\right)^{1/6}\).
	
	Consequently, the attack also succeeds whenever there exists a
	connecting isogeny \(\varphi\colon E\longrightarrow E^{(p)}\)
	such that \(\delta_E < \deg(\varphi)
	\leq B\left(\frac{p}{2}\right)^{1/3}\) and \(\Pplus(\deg(\varphi))\leq B\).
	Thus, the attack is not restricted to the case in which the minimal
	connecting degree \(\delta_E\) is \(B\)-smooth.
	
	More precisely, the relevant success event is not merely \(\Pplus(\delta_E)\leq B\),
	but rather
	\[
	\left\{
	\exists\,\varphi\colon E\to E^{(p)}
	\;:\;
	\Pplus(\deg(\varphi))\leq B
	\ \text{and}\
	\deg(\varphi)\leq
	B\left(\frac{p}{2}\right)^{1/3}
	\right\}.
	\]
	Accordingly, the success probability is
	\[
	\Prob\left(
	\exists\,\varphi\colon E\to E^{(p)}
	\;:\;
	\Pplus(\deg(\varphi))\leq B
	\ \text{and}\
	\deg(\varphi)\leq
	B\left(\frac{p}{2}\right)^{1/3}
	\right).
	\]
	
	The Frobenius-claw event considered in~\cite{Mamah2026} corresponds
	precisely to this broader event. Therefore, after disregarding
	finite \(p\) deviations from the limiting distribution, the success
	probability of the claw attack is at least \(\Prob\bigl(\Pplus(\delta_E)\leq B\bigr),
	\)
	and may be strictly larger because the attack can also exploit
	non-minimal connecting isogenies.
	
	To capture this larger success event, we may omit the restriction
	encoded by \(P_E\) and work directly with \(S_E\). Indeed, \(S_E\)
	counts the connecting isogenies whose degrees are \(B\)-smooth and lie
	in the interval \([X/2,X]\). The previous choice of
	\(\cA(X,B)\) was tailored to facilitate the analysis of \(P_E\).
	Since minimality is no longer required, we instead introduce the
	cumulative set
	\[
	\cA_{\leq}(X,B)
	:=
	\left\{
	n\in\Z_{\geq 2}
	:
	n\leq X,\;
	\mu^2(n)=1,\;
	\Pplus(n)\leq B
	\right\}.
	\]
	Repeating the preceding argument with \(\cA(X,B)\) replaced by
	\(\cA_{\leq}(X,B)\) yields the following sharper lower bound for the
	success probability.

	As a consequence, combining with \Cref{thm:squarefree-pair-bound}, we have 
	\begin{equation}
		C_p(m,n)=O\left( \sqrt{mn}(1+\log 2mn)^2\right)
	\end{equation}
	Let \(S_E^{\leq}
	:=
	\sum_{n\in\mathcal A_{\leq}(X,B)}
	r_E^\circ(n)\). 
	
	
	\begin{proposition}
		Denote $\Prob(A)$ the probability of existing isogeny $\varphi\colon E\to E^{(p)}$ satisfying that $\deg \varphi\leq B(p/2)^{1/3}$ and $\deg \varphi$ is $B$-smooth.Then  \begin{equation}\label{eq:probofall}
			\Prob(A) \ge  \bigl(\frac{4\sqrt{2}\cdot 0.655 e^{-2}}{\pi\zeta(2)} +o_{\eps}(1) \bigr)\frac{\rho(u)}{\log p}.
		\end{equation}
	\end{proposition}
	\begin{proof}
		Since
		\(\mathbf{1}_{\{s\ge 1\}} \ge s - \binom{s}{2}\) is right for all integer,
		\[
		\mathbb P(A) \ge \Prob(S_E^{\leq} \ge 1) 
		\ge \EE[S^{\leq}_E] - \EE\left[\binom{S_E^{\leq}}{2}\right].\]
		Since every $n\in\mathcal A_{\le}(X,B)$ is
		squarefree, \Cref{thm:squarefree-pair-bound} and \Cref{lem:uniform-divisor-sum} give
		\[
		C_p(m,n)=O\left(
		\sqrt{mn}\bigl(1+\log(2mn)\bigr)^2\right)
		\]
		for all $m,n\in\mathcal A_{\le}(X,B)$. Let \[R_{p,B,X}=\frac{\EE\left[\binom{S_E^{\leq}}{2}\right]}{\EE[S^{\leq}_E]}\]
		If $R_{p,B,X}\to 0$ as $p\to \infty$, we have $\Prob(A) \geq (1-o(1))\EE[S_E^{\geq}]$. Choose $X=N-1$ we have \eqref{eq:probofall} .
		
		Now we prove the limitation $R_{p,B,X}$ tends to $0$:
		
		
		
		

		\begin{align*}
			\EE\left[\binom{S_E^\le}{2}\right]
			&=
			\frac{1}{h_p}
			\left(
			\sum_{\substack{m,n\in\mathcal A_\le(X,B)\\m<n}}
			C_p(m,n)
			+\frac12\sum_{n\in\mathcal A_\le(X,B)}C_p(n,n)
			\right)\\
			&\leq
			\frac{C_1(1+\log(2X))^2}{h_p}
			\left(
			\sum_{n\in\mathcal A_\le(X,B)}\sqrt n
			\right)^2\\
			&\leq
			C_2\frac{X^3}{p}\rho(v)^2(\log p)^2,
		\end{align*}
		where $C_1,C_2>0$ are absolute constants.
		
		On the other hand, since
		$\cA(X,B)\subseteq\mathcal A_\le(X,B)$,
		\begin{align*}
			\EE[S_E^\le]
			&=
			\frac{1}{h_p}
			\sum_{n\in\mathcal A_\le(X,B)}
			A_p^{\mathrm{BH}}(n)\\
			&\geq
			\frac{1}{h_p}
			\sum_{n\in\cA(X,B)}
			A_p^{\mathrm{BH}}(n)\\
			&\geq
			\frac{0.655e^{-2}}{\pi h_p\log p}
			\sqrt{\frac{pX}{2}}
			\bigl(\#\cA(X,B)-1\bigr)\\
			&\geq
			C_3\frac{X^{3/2}\rho(v)}
			{\sqrt p\log p},
		\end{align*}
		where $C_3>0$ is
		an absolute constant. The last inequality
		follows from \Cref{lem:squarefree-shell} and $h_p\asymp p$.
		
		Combining with the upper bound of $\EE\left[\binom{S_E^{\leq}}{2}\right]$, we obtain
		\begin{equation}\label{eq:R-bound}
			R_{p,B,X}
			=
			O\left(
			\frac{X^{3/2}}{\sqrt p}\,
			\rho(v)(\log p)^3
			\right).
		\end{equation}
		
		Now choose $X=N-1$, where $N=\lfloor(p/2)^{1/3}\rfloor$.
		Then \(\frac{X^{3/2}}{\sqrt p}
		\longrightarrow\frac1{\sqrt2}\).
		Since \(
		\log\rho(v)=-(1+o(1))v\log v\) and $v\geq (\log \frac{p}{2})^\varepsilon/2$, we have \(\rho(v)(\log p)^3=o(1)\).
		Consequently, \(R_{p,B,X}=o(1)\).
		
		Therefore,
		\begin{align*}
			\Prob(A)
			&\geq \Prob(S_E^\le\geq1)\\
			&\geq
			\EE[S_E^\le]
			-
			\EE\left[\binom{S_E^\le}{2}\right]\\
			&=(1-R_{p,B,X})\EE[S_E^\le]\\
			&\geq
			C_\varepsilon\frac{\rho(u)}{\log p},
		\end{align*}
		for some constant $C_\varepsilon>0$
		depending only on $\varepsilon$.
		Here we used $\rho(v)\geq\rho(u)$,
		since $v\leq u$.
	\end{proof}
	As a result, the success probability satisfies
	\(\Prob(A)\geq C_{\varepsilon}\frac{\rho(u)}{\log p}\), 
	for some constant \(C_{\varepsilon}>0\) depending only on \(\varepsilon\).

	
	\newcommand{\etalchar}[1]{$^{#1}$}
	\providecommand{\bysame}{\leavevmode\hbox to3em{\hrulefill}\thinspace}
	\providecommand{\MR}{\relax\ifhmode\unskip\space\fi MR }
	\providecommand{\MRhref}[2]{%
		\href{http://www.ams.org/mathscinet-getitem?mr=#1}{#2}
	}

\providecommand{\href}[2]{#2}
\providecommand{\etalchar}[1]{$^{#1}$}
\providecommand{\MR}[1]{MR~#1}

\newpage
\appendix
\section{Proof of the lemmas in Section~\ref{sec:notation}.}\label{ap:proofoflemmas}

	\begin{proof}[Proof of Lemma~\ref{lem:Psi-uniform}] 
	Hildebrand's uniform estimate\cite[Theorem~1, pp.~289--290]{Hildebrand1986} states that for each fixed \(\eta>0\),
	\[ 
	\frac{\Psi(x,y)}{x}
	=
	\rho(v)
	\left\{
	1+
	O_\eta\left(
	\frac{\log(v+1)}{\log y}
	\right)
	\right\}\]
	where \(v=\frac{\log x}{\log y}\), whenever
	\(2\leq y\leq x\) and \(\log y\geq(\log\log x)^{5/3+\eta}\).

	We apply this result with
	\(x=N_0,
	y=B,
	v=u,\eta=\frac{\eps}{2}\). To apply correctly, we first prove $\log B\ge (\log\log N_0)^{5/3+\eta}$. Since
	\(\log\log N_0
	=
	\log\log p+O(1)\), and $\log B= \frac{\log N_0}{u}\ge C_{0}\frac{\log p}{(\log p)^{1-\eps}}= C_0(\log p)^\eps$ for some $C_0$,  we have that there exists constant $C_0'$ 
	\[
	\frac{\log B} 
	{(\log\log N_0)^{5/3+\eps/2}} \ge \frac{C_0(\log p)^\varepsilon}
	{(\log\log N_0)^{5/3+\eps/2}}=\frac{C_0'(\log p)^\varepsilon}
	{(\log\log p)^{5/3+\eps/2}}
	\longrightarrow\infty \text{ as }p\to \infty
	\]
	
	Therefore, we apply this result and obtain 
	\[
	\frac{\Psi(N_0,B)}{N_0}
	=
	\rho(u)
	\left\{
	1+
	O_\varepsilon\left(
	\frac{\log(u+1)}{\log B}
	\right)
	\right\}.\]

	Since \(N=\lfloor N_0\rfloor\) we have
	\(\Psi(N,B)=\Psi(N_0,B)\). Then we obtain
	\[
	\frac{\Psi(N,B)}{N}
	=
	\frac{N_0}{N}
	\frac{\Psi(N_0,B)}{N_0}
	=
	\rho(u)
	\left\{
	1+
	O_\varepsilon\left(
	\frac{\log(u+1)}{\log B}
	+
	\frac1N
	\right)
	\right\}.\] 
	Using the upper bound \(u<(\log p)^{1-\varepsilon}\) and $\log B\ge C_0(\log p)^{\eps}$, we obtain as $p\to \infty$:  
	\[  \frac{\log(u+1)}{\log B}
	\leq \frac{(1-\eps)\log \log p}{C_0(\log p)^{\eps}}
	\longrightarrow0. \]
\end{proof}

	\begin{proof}[Proof of Lemma~\ref{lem:squarefree-shell}] 
	For counting $\cA(X,B)$, we construct the following $\Psi_1$ counting $B$-smooth squarefree integers in \([1,x]\): 
	\[
	\Psi_1(x,B;\mu^2)=\sum_{\substack{n\le x\\ P^+(n)\le B}} \mu(n)^2.
	\]
	Then $\#\cA(X,B)=\Psi_1(X,B;\mu^2)-\Psi_1(\frac{X}{2},B;\mu^2)$.
	
	From the theorem of Ivić and Tenenbaum\cite[Theorem~1 and Corollary~(2.12)]{IvicTenenbaum1986} , we have
	\[
	\Psi_1(x,B;\mu^2)=\zeta(2\beta(x,B))^{-1}\Psi(x,B)\left\{1+O_\delta\left(\frac{\log\log B}{\log B}\right)\right\},
	\]
	where \(\beta(x,B)=1-\frac{\xi(v_x)}{\log B}\) with $v_x=\frac{\log x}{\log B}$ and \(\xi(w)\) is determined by
	\(e^{\xi(w)}=1+w\xi(w)\).

	The standard estimate \(\xi(w)=O(\log(w+1))\) gives \(1-\beta(x,B)=O\left(\frac{\log(v_x+1)}{\log B}\right)\). That is to say $\beta(x,B)=1+O\left(\frac{\log (v_x+1)}{\log B}\right)$. By definition \(v_x=\frac{\log x}{\log B}\leq \frac{\log X}{\log 2}\)
	and therefore \(\log(v_x+1)=O(\log\log X)\). By condition \(\log B\geq(\log\log X)^{5/3+\delta}\), we have
	\[
	\frac{\log(v_x+1)}{\log B}=O\left(
	(\log\log X)^{-2/3-\delta}\right)=o(1).
	\]
	Thus \(\beta(x,B)\to 1\) uniformly in the stated parameter range. We only need the value of $\Psi_1$ at $x=X$ and $x=\frac{X}{2}$, consider $x=X$ and $x=\frac{X}{2}$ specially, we have $\beta(X,B)\to1$ and $\beta(\frac{X}{2},B)\to 1$. 
	
	Hence, we obtain
	\[
	\Psi_1(X,B;\mu^2)=\frac{1}{\zeta(2)}\Psi(X,B);\quad \Psi_1\left(\frac{X}{2},B;\mu^2\right)=\frac{1}{\zeta(2)}\Psi\left(\frac{X}{2},B\right).
	\]
	
	Moreover, the theorem of Hildebrand and Tenenbaum\cite[Theorem~3]{HildebrandTenenbaum1986} yields
	\begin{equation}
		\label{eq:localPsiQuation}
		\frac{\Psi(X/2,B)}{\Psi(X,B)}
		=2^{-\alpha(X/2,B)}
		\left\{1+O\left(\frac1v+\frac{\log B}{B}\right)\right\}=2^{-\alpha(X/2,B)}
		\left\{1+o_{\delta}(1)\right\},
	\end{equation}
	where  \(\alpha(x,B)\) denote the unique positive solution of
	\begin{equation}\label{eq:alpha-def}
		\sum_{\substack{\ell\leq B\\ \ell\ {\rm prime}}}
		\frac{\log\ell}{\ell^{\alpha(x,B)}-1}
		=
		\log x.
	\end{equation}
	The saddle-point estimate\cite[Lemma~2, Equation~(3.5)]{HildebrandTenenbaum1986}
	gives \( 1-\alpha(x,B)=O\left( \frac{\log(v_x+1)}{\log B}\right)\). So \(\alpha(x,B)\to 1\) uniformly as well. As a consequence,
	\[
	\frac{\Psi(X/2,B)}{\Psi(X,B)}=\frac12\left\{1+O\left(\frac1v+\frac{\log B}{B}\right)\right\}=\frac12\{1+o_\delta(1)\}.
	\]

	Therefore
	\[
	\#\mathcal{A}(X,B)=\Psi_1(X,B;\mu^2)-\Psi_1\left(\frac{X}{2},B;\mu\right)=\left(\frac{1}{2\zeta(2)}+o_\delta(1)\right)\Psi(X,B).
	\]
	
	\Cref{lem:Psi-uniform} applies and gives
	\[
	\Psi(X,B)=X\rho(v)\big(1+o_\delta(1)\big),
	\]
	which complete the proof.
\end{proof}

\begin{proof}[Proof of Lemma~\ref{lem:tatuzawa}]          
	Since \(n\) is squarefree and \(n<p\) for $n\in \mathcal{A}(X, B)$, $X<p/4$, the integer \(pn\) is
	squarefree. The quadratic discriminant formula therefore gives
	\eqref{eq:Dn-def}. The map \(n\mapsto D_n\) is injective:
	injectivity within each case is immediate, while equality between
	the two cases would require one squarefree integer to be four
	times another.
	
	Since \(n>1\) and \(X<p/4\), we have
	\begin{equation}\label{eq:Dn-range}
		p<|D_n|\leq4pX<p^2.
	\end{equation}
	
	For any fixed \(0<\eta<1/2\), Tatuzawa's theorem
	\cite[Theorem~2]{Tatuzawa1951} states that, apart from at most
	one exceptional fundamental discriminant \(D<0\), every primitive
	real Dirichlet character \(\chi_D\) with
	\(
	|D|>\max\{e^{1/\eta},e^{11.2}\}\) satisfies
	\(
	L(1,\chi_D)>0.655\,\eta |D|^{-\eta}\).
	
	Set \(\eta=1/\log p\) and assume \(p>e^{11.2}\).
	By \eqref{eq:Dn-range},
	\(|D_n|>p=e^{1/\eta}>e^{11.2}\). 
	Hence, apart from at most one exceptional discriminant,
	\[
	L(1,\chi_{D_n})
	>0.655\,\eta |D_n|^{-\eta}
	>\frac{0.655e^{-2}}{\log p}.
	\]
	Since the map \(n\mapsto D_n\) is injective, the exceptional
	discriminant corresponds to at most one
	\(n\in\mathcal A(X,B)\). This proves
	\eqref{eq:tatuzawa-primitive}.
	
	If \(pn\equiv3\pmod4\), then \(D_n=-pn\), and the character \(\chi_{-4pn}=\chi_{4D_n}\) is induced from \(\chi_{D_n}\).  Hence
	\[
	L(1,\chi_{-4pn})
	=
	L(1,\chi_{D_n})
	\left(1-\frac{\chi_{D_n}(2)}{2}\right)
	\geq\frac12L(1,\chi_{D_n})
	>
	\frac{0.655e^{-2}}{2\log p},
	\]
	as required.
\end{proof}


\begin{proof}[Proof of Lemma~\ref{lem:uniform-divisor-sum}]
	For any positive integer $k$, pairing each divisor $d\mid k$ with $k/d$ yields
	\[
	\tau(k)\leq
	2\sum_{\substack{d\mid k\\d\leq\sqrt{k}}}1.
	\]
	When $t^2<4M$, we have $0<4M-t^2\leq 4M$. We may therefore widen the divisor range and interchange the order of summation to obtain
	\begin{align}
		\sum_{\substack{t\in\mathbb Z\\t^2<4M}}
		\tau(4M-t^2)
		&\leq
		2\sum_{\substack{t\in\mathbb Z\\t^2<4M}}
		\sum_{\substack{d\mid(4M-t^2)\\d\leq\sqrt{4M}}}1
		\notag\\
		&=
		2\sum_{d\leq\sqrt{4M}}
		\#\left\{
		t\in\mathbb Z:
		|t|<\sqrt{4M},\
		t^2\equiv4M\pmod d
		\right\}.
		\label{eq:divisor-reduction}
	\end{align}
	
	For each positive integer $d$, let
	\[
	\nu_{4M}(d)
	:=
	\#\{x\bmod d:x^2\equiv 4M\pmod d\}.
	\]
	Each residue class modulo $d$ contains at most
	$2\sqrt{4M}/d+1$ integers in the interval
	$[-\sqrt{4M},\sqrt{4M}]$. Since $d\leq\sqrt{4M}$,
	we have
	\[
	\begin{aligned}
		&\#\left\{
		t\in\mathbb Z:
		|t|<\sqrt{4M},\
		t^2\equiv 4M\pmod d
		\right\}\\
		&\qquad\leq
		\left(\frac{2\sqrt{4M}}{d}+1\right)\nu_{4M}(d)
		\leq
		\frac{3\sqrt{4M}}{d}\nu_{4M}(d).
	\end{aligned}
	\]
	Consequently, by \eqref{eq:divisor-reduction},
	\begin{equation}\label{eq:divisor-reduction-2}
		\sum_{\substack{t\in\mathbb Z\\t^2<4M}}
		\tau(4M-t^2)
		\leq
		6\sqrt{4M}
		\sum_{d\leq\sqrt{4M}}\frac{\nu_{4M}(d)}{d}.
	\end{equation}
	
	We next estimate the sum on the right hand side.
	By the Chinese remainder theorem, $\nu_{4M}$
	is multiplicative. We therefore first derive bounds for $\nu_{4M}(q^j)$ for prime $q$, and then assemble these local
	estimates via an Euler product.

	Fix a prime $q$, and write
	$e=v_q(4M)$, i.e. $4M=q^e a$ where $q\nmid a$. We estimate $\nu_{4M}(q^j)$ for each $j\geq1$.
	
	First, suppose that $1\leq j\leq e$.
	Since $q^j\mid 4M$, the congruence
	$x^2\equiv4M\pmod{q^j}$ reduces to
	$x^2\equiv0\pmod{q^j}$, which is
	equivalent to
	$q^{\lceil j/2\rceil}\mid x$.
	Consequently, \(\nu_{4M}(q^j)=q^{\lfloor j/2\rfloor}\).
	
	Now suppose that $j>e$. If $e$ is odd,
	the congruence
	\(x^2\equiv q^e a\pmod{q^j}\) forces $v_q(x^2)=e$, it is impossible because $v_q(x^2)=2v_q(x)$ is even. Hence
	\(\nu_{4M}(q^j)=0\) if $j>e$ and $e$ is odd.
	
	It remains to treat the case
	$e=2k$. Write $j=2k+r$, where $r\geq1$.
	The congruence
	\(
	x^2\equiv q^{2k}a\pmod{q^{2k+r}}\) implies $q^k\mid x$. Setting $x=q^ky$,
	we obtain
	\(
	y^2\equiv a\pmod{q^r}.\) Thus
	\(
	\nu_{4M}(q^{2k+r})
	=q^k\nu_a(q^r)\), where $\nu_a(d)=\#\{y\bmod d:y^2\equiv a\bmod d\}$ 
	
	We next bound $\nu_a(q^r)$.
	Since $q\nmid a$, the polynomial
	$x^2-a$ has at most two roots in
	$\mathbb F_q$. For odd $q$, every
	such root is simple, so
	Hensel's lemma gives
	\(
	\nu_a(q^r)\leq2\) when $q$ odd and $r\geq1$. Therefore,
	\[
	\sum_{r\geq1}\frac{\nu_a(q^r)}{q^r}
	\leq
	2\sum_{r\geq1}q^{-r}
	=\frac{2}{q-1}.
	\]
	
	For $q=2$, since $a$ is odd, we have
	\(
	\nu_a(2)\leq1\) and \(\nu_a(4)\leq2\). When $r\geq3$, the congruence
	$x^2\equiv1\pmod{2^r}$ has at most
	four solutions. Indeed,
	$2^r\mid(x-1)(x+1)$ implies
	\(2^{r-1}\mid(x-1)\) or \(2^{r-1}\mid(x+1),\) because $\gcd(x-1,x+1)=2$ for odd $x$.
	Thus $x\equiv\pm1\pmod{2^{r-1}}$,
	yielding at most four residue classes
	modulo $2^r$.
	
	If $x_0$ solves $x^2\equiv a\pmod{2^r}$, then the congruence is equivalent to $(x_0^{-1}x)^2\equiv 1\pmod{2^r}$.
	Consequently,
	\(
	\nu_a(2^r)\leq4\) when $r\geq3$. It follows that
	\[
	\sum_{r\geq1}\frac{\nu_a(2^r)}{2^r}
	\leq
	\frac12+\frac24+
	4\sum_{r\geq3}2^{-r}
	=2.
	\]
	
	Combining cases above, we obtain
	\begin{equation}\label{eq:unit-root-sum}
		\sum_{r\geq1}\frac{\nu_a(q^r)}{q^r}
		\leq\frac{2}{q-1}
	\end{equation}
	for every prime $q$ with $q\nmid a$.
	
	We put these estimates together.
	If $e=2k$, then
	\begin{align*}
		\sum_{j\geq1}
		\frac{\nu_{4M}(q^j)}{q^j}
		&=
		2\sum_{h=1}^{k}\frac1{q^h}
		+
		q^{-k}\sum_{r\geq1}
		\frac{\nu_a(q^r)}{q^r}\\
		&\leq
		\frac{2(1-q^{-k})}{q-1}
		+
		\frac{2q^{-k}}{q-1}=\frac2{q-1}.
	\end{align*}
	If $e=2k+1$, then
	\begin{align*}
		\sum_{j\geq1}
		\frac{\nu_{4M}(q^j)}{q^j}
		&=
		2\sum_{h=1}^{k}\frac1{q^h}
		+\frac1{q^{k+1}}\\
		&\leq\frac2{q-1}.
	\end{align*}
	Therefore, uniformly in $M$ and $q$,
	\begin{equation}\label{eq:uniform-local-factor}
		1+\sum_{j\geq1}
		\frac{\nu_{4M}(q^j)}{q^j}
		\leq
		1+\frac2{q-1}
		\leq
		(1-q^{-1})^{-2}.
	\end{equation}
	
	By multiplicativity and nonnegativity of
	$\nu_{4M}$, we obtain
	\begin{equation}
		\begin{aligned}
			\sum_{d\leq\sqrt{4M}}\frac{\nu_{4M}(d)}{d}
			&=
			\sum_{d\leq\sqrt{4M}}
			\prod_{q\mid d}
			\frac{\nu_{4M}(q^{v_q(d)})}{q^{v_q(d)}}\\
			&\leq
			\prod_{q\leq\sqrt{4M}}
			\left(
			1+\sum_{j\geq1}
			\frac{\nu_{4M}(q^j)}{q^j}
			\right)\\
			&\leq
			\prod_{q\leq\sqrt{4M}}
			(1-q^{-1})^{-2}\\
			&\leq C_{Mer}\bigl(1+\log\sqrt{4M}\bigr)^2,
		\end{aligned}
	\end{equation}
	where $C_{Mer}>0$ is an absolute constant. The last
	inequality follows from Mertens' product estimate.
	
	Substituting this bound into
	\eqref{eq:divisor-reduction-2}, we conclude that
	\[
	\sum_{\substack{t\in\mathbb Z\\t^2<4M}}
	\tau(4M-t^2)=O\left(
	\sqrt M\bigl(1+\log(2M)\bigr)^2\right),
	\]
	with an absolute constant independent of $M$.
\end{proof}

\begin{proof}[Proof of Lemma~\ref{lem:joint-divisor-sum}] 
		Put
	\(Q(n)=
	\sum_{m\leq n}
	\sum_{t^2<4mn}\tau(4mn-t^2)\).
	For every positive integer \(k\), there is
	\[
	\tau(k)\leq2\sum_{\substack{d\mid k\\d\leq\sqrt{k}}}1.
	\]
	Since \(0<4mn-t^2\leq4n^2\), every divisor occurring on the right is at most \(2n\). 
	\begin{equation}\label{eq:Qn-congruence}
		Q(n)\leq 2\sum_{m\leq n}\sum_{t^2<4mn}\sum_{d\mid 4mn-t^2\\d\leq \sqrt{4mn-t^2}}1\leq2\sum_{d\leq2n}N_d,
	\end{equation}
	where
	\(N_d=
	\#\left\{(m,t):1\leq m\leq n,\ |t|<2n,\ 4nm\equiv t^2\pmod d\right\}\). 
	
	Fix \(d\), and put \(g=(4n,d)\).  The linear congruence in \(m\) is soluble if and only if \(g\mid t^2\).  Define
	\[
	h(g)=\prod_{q^e\parallel g}q^{\ceil{e/2}}.
	\]
	Then \(g\mid t^2\) if and only if \(h(g)\mid t\).  Since \(g\mid4n\) and \(n\) is squarefree, every odd prime occurs in \(g\) to exponent at most one and \(v_2(g)\leq3\).  Hence \(\frac{g}{h(g)}\leq2\).
	It follows that the number of integers \(t\) with \(|t|<2n\) and \(g\mid t^2\) is at most
	\(
	\frac{4n}{h(g)}+1
	\leq\frac{8n}{g}+1
	\leq\frac{10n}{g}\),
	where the last inequality uses \(g\leq d\leq2n\).
	
	For each such \(t\), division by \(g\) gives
	\(
	\frac{4n}{g}m\equiv\frac{t^2}{g}\pmod{d/g}\), 
	and the coefficient \(4n/g\) is invertible modulo \(d/g\).  Thus \(m\) lies in one residue class modulo \(d/g\), so there are at most \(\frac{ng}{d}+1\) choices with \(1\leq m\leq n\).  Therefore
	\begin{align*}
		N_d
		\leq\frac{10n}{g}\left(\frac{ng}{d}+1\right)
		\leq30\frac{n^2}{d},
	\end{align*}
	where we used \(d\leq2n\) in the last step.  Substitution into \eqref{eq:Qn-congruence} gives
	\[
	Q(n)\leq60n^2\sum_{d\leq2n}\frac1d
	\leq60n^2\bigl(1+\log(2n)\bigr).
	\]
\end{proof}

\section{Mamah's heuristic}\label{ap:Mhuristic}
Mamah~\cite{Mamah2026} stated another heuristic in his analysis of the effectiveness of Wesolowski's algorithm, this heuristic originally commes from \cite{BooherEtAl2024}: 
\begin{heurm}[Heuristic 1, \cite{Mamah2026}]  For a fixed degree $n$, the probability that a supersingular curve $E$
	admits an $n$-isogeny to its Frobenius conjugate $E^{(p)}$ is
	\[ \Pr(A_{E,n}) \approx \sqrt{\frac{12n}{p}} .  \]	
\end{heurm}

For $1\leq n<p/4$, define
\[
\mathsf A_{E,n}
:=
\{r_E>0\},
\]
where $E$ is sampled uniformly from $\Sp$. Then Heuristic M is \begin{equation}
	\label{eq:their-heuristic-one}
	\Prob(\mathsf{A}_{E,n})
	\approx
	\sqrt{\frac{12n}{p}},
\end{equation}  
This heuristic also consider $m,n$ with square divisors, so we derive a non-squarefree bound for $C_{p}(m,n)$ as below:

\begin{lemma}
    \label{lem:general-mixed-pair-bound}
    Let $p\geq5$ be prime and let $1\leq m,n<p/4$.
    Put
    \[
    s(k):=\prod_{\substack{q\mid k\\q\text{ prime}}}
    q^{\lfloor v_q(k)/2\rfloor},
    \]
    so that $s(k)^2$ is the largest square divisor of $k$.
    Then
    \begin{equation}\label{eq:general-mixed-pair-bound}
        C_p(m,n)
        =
        O\left(
        \sqrt{mn}\,s(\gcd(m,n))
        \bigl(1+\log(2mn)\bigr)^2
        \right),
    \end{equation}
    where the implied constant is absolute.
    In particular,
    \[
    C_p(n,n)
    =
    O\left(
    n^{3/2}\bigl(1+\log(2n)\bigr)^2
    \right).
    \]
\end{lemma}

\begin{proof}
    Fix an integer $t$ arising from an ordered pair of distinct primitive connection-isogeny orbits
    of degrees $m$ and $n$, respectively, and let $R$ be the associated quaternion order. Put
    \[
    D:=4mn-t^2>0,\qquad
    e_q:=v_q(D),\qquad
    g_q:=\min\{v_q(m),v_q(n)\}.
    \]
    By \Cref{lem:pairs-embeddings}, we bound
    $C_{p,t}(m,n)$ by counting $N_{\max}(R)$.

    At $q=p$, the division algebra
    $(B_{p,\infty})_p$ has a unique maximal order,
    so $\#\cB_p(R)=1$
    \cite[Proposition~13.3.4 and
    Theorem~13.3.11(b)]{Voight2021}.

    Fix a finite prime $q\neq p$. Since
    $B_{p,\infty}$ is ramified only at $p$
    and $\infty$, one has \(
    B_q\simeq M_2(\Q_q)\). The maximal $\Z_q$-orders in $B_q$ are
    naturally identified with the vertices
    of the $(q+1)$-regular Bruhat--Tits tree
    of $\mathbf{PGL}_2(\Q_q)$; see
    \cite[Definition~23.5.5 and
    Proposition~23.5.8]{Voight2021}.
    Under this identification, the set
    $\cB_q(R)$ determines a convex subtree.

    If two vertices in $\cB_q(R)$ are at
    distance $d$, their intersection is an
    Eichler order of level $q^d$.
    Since $R_q$ is contained in this
    intersection, we have
    \(
    d\leq v_q(\operatorname{discrd}(R))=e_q\). Thus
    \[
    \operatorname{diam}\cB_q(R)\leq e_q.
    \]
    In particular, if $e_q=0$, then
    $\cB_q(R)$ contains at most one vertex.

    We next estimate the depth of $\cB_q(R)$.
    Let
    \(
    A_i=\Z_q[i],\ K_i=\Q_q[i]\). Then $K_i$ is either a quadratic field
    extension of $\Q_q$ or the split algebra
    $\Q_q\times\Q_q$.
    Let $\mathcal O_{K_i}$ denote its maximal
    $\Z_q$-order. There exists an integer
    $f_i\geq0$ such that
    \(A_i=\Z_q+q^{f_i}\mathcal O_{K_i}\).
   
    Since $i^2=-pm$, the discriminant-conductor
    formula gives
    \(
    (-4pm)\Z_q
    =
    \operatorname{disc}(A_i)
    =
    q^{2f_i}\operatorname{disc}(\mathcal O_{K_i}).\) 
    Hence
    \[
    2f_i\leq v_q(4pm).
    \]
    Similarly, for
    \(
    A_j=\Z_q[j]
    =\Z_q+q^{f_j}\mathcal O_{K_j}\),
    where $j^2=-pn$, we obtain
    \(
    2f_j\leq v_q(4pn)\).

    Set $f_q:=\min\{f_i,f_j\}$.
    Since $p$ is prime integer, we obtain
    \begin{equation}\label{eq:mixed-fq-bound}
        f_q\leq
        \begin{cases}
            \left\lfloor g_q/2\right\rfloor,
            &q\neq2,p,\\[1ex]
            1+\left\lfloor g_2/2\right\rfloor,
            &q=2.
        \end{cases}
    \end{equation}

    Since $A_i,A_j\subseteq R_q$, we have
    \(
    \cB_q(R)
    \subseteq
    \cB_q(A_i)\cap\cB_q(A_j).
    \) Let $S_i=\cB_q(\mathcal O_{K_i})$ and $S_j=\cB_q(\mathcal O_{K_j})$. The standard branch
    description of quadratic orders
    \cite[\S1]{ArenasCarmonaSaavedra2016}
    gives
    \[
    \cB_q(A_i)\subseteq N_{f_i}(S_i),
    \qquad
    \cB_q(A_j)\subseteq N_{f_j}(S_j),
    \]
    where $N_f(S)$ denotes the closed
    $f$-neighborhood of $S$.
    Consequently, choosing the stem
    corresponding to the smaller conductor
    exponent $f_q$ to be $S$, we obtain
    \[
    \cB_q(R)\subseteq N_{f_q}(S).
    \]

    The stem $S$ has one of three possible
    shapes: an apartment if the corresponding
    quadratic algebra is split, a single
    vertex if it is an unramified quadratic
    field extension, or a single edge if
    it is a ramified quadratic field
    extension; see
    \cite[\S1]{ArenasCarmonaSaavedra2016}.

    We now translate these geometric descriptions
    into a uniform vertex bound.
    If $S$ is an apartment, put $\Lambda_q=S$.
    Otherwise, extend the vertex or edge $S$
    to an apartment $\Lambda_q$.
    In either case, $S\subseteq\Lambda_q$, and
    therefore
    \(
    \cB_q(R)
    \subseteq N_{f_q}(S)
    \subseteq N_{f_q}(\Lambda_q).\)

    Since nearest-point projection onto
    $\Lambda_q$ does not increase distances,
    the projection of $\cB_q(R)$ contains
    at most $e_q+1$ vertices.
    For each vertex of $\Lambda_q$, the
    corresponding projection fiber inside
    $N_{f_q}(\Lambda_q)$ contains
    \[
    1+(q-1)\sum_{r=0}^{f_q-1}q^r=q^{f_q}
    \]
    vertices, with the sum interpreted
    as empty when $f_q=0$.
    Hence
    \[
    \#\cB_q(R)\leq(e_q+1)q^{f_q}.
    \]

    Combining this with \eqref{eq:mixed-fq-bound},
    we obtain
    \[
    \#\cB_q(R)
    \leq
    (e_q+1)q^{\lfloor g_q/2\rfloor}
    \qquad(q\neq2,p),
    \]
    and
    \[
    \#\cB_2(R)
    \leq
    2(e_2+1)2^{\lfloor g_2/2\rfloor}.
    \]

    Since $m,n<p$, we have $p\nmid\gcd(m,n)$.
    The local-global product correspondence yields
    \[
    \begin{aligned}
    N_{\max}(R)
    &=\prod_q\#\cB_q(R)\\
    &\leq
    2\left(
    \prod_{q\neq p}
    q^{\lfloor g_q/2\rfloor}
    \right)
    \left(
    \prod_{q\neq p}(e_q+1)
    \right)\\
    &\leq
    2s(\gcd(m,n))\tau(D).
    \end{aligned}
    \]
    Combining this with
    \Cref{lem:pairs-embeddings}, we obtain
    \[
    \begin{aligned}
    C_{p,t}(m,n)
    &\leq48N_{\max}(R)\\
    &\leq
    96s(\gcd(m,n))
    \tau(4mn-t^2).
    \end{aligned}
    \]
    Summing over all admissible $t$, we have
    \[
    C_p(m,n)
    \leq
    96s(\gcd(m,n))
    \sum_{\substack{t\in\Z\\t^2<4mn}}
    \tau(4mn-t^2).
    \]
    Applying \eqref{eq:uniform-divisor-sum}
    with $M=mn$, we obtain
    \[
    \sum_{\substack{t\in\Z\\t^2<4mn}}
    \tau(4mn-t^2)
    =
    O\left(
    \sqrt{mn}\bigl(1+\log(2mn)\bigr)^2
    \right).
    \]
    Consequently,
    \[
    C_p(m,n)
    =
    O\left(
    \sqrt{mn}\,s(\gcd(m,n))
    \bigl(1+\log(2mn)\bigr)^2
    \right).
    \]
    Finally, taking $m=n$ and using
    $s(n)\leq\sqrt n$, we conclude that
    \[
    C_p(n,n)
    =
    O\left(
    n^{3/2}\bigl(1+\log(2n)\bigr)^2
    \right).
    \qedhere
    \]
\end{proof}

Recall that the number of all connection isogenies bewteen $E$ and $E^{(p)}$ with degree $n$ $r_E(n)=\sum\limits_{d^2\mid n}r^{\circ}_E(\frac{n}{d^2})$. By \Cref{prop:primitive-mass}, we have 
\[A_{p}^{BH}(n)=\sum\limits_{E\in\Sp}r^{\circ}_E(n)=\frac{\sqrt{pn}}{\pi}L(1,\chi_{-4pn})
+
\frac{\sqrt{pn}}{2\pi}\mathbf1_{\{pn\equiv3\pmod4\}}
L(1,\chi_{-pn})\] 
Set
\(
\mathcal L_{p,n}
:=
L(1,\chi_{-4pn})
+
\frac12\mathbf1_{\{pn\equiv3\pmod4\}}
L(1,\chi_{-pn})\) and \(
\mathcal H_{p,n}
:=
\sum\limits_{d^2\mid n}
\frac{\mathcal L_{p,n/d^2}}d\), we have that \[\sum\limits_{E\in\Sp}r^{\circ}_E=A_p^{BH}=\frac{\sqrt{np}}{\pi}\mathcal{L}_{n,p},\quad \sum\limits_{E\in \Sp}r_E(n)=\sum\limits_{E\in \Sp}\sum\limits_{d^2\mid n}r^{\circ}_E\left(\frac{n}{d^2}\right)=\frac{\sqrt{np}}{\pi}\mathcal H_{n,p}.\]

\begin{theorem}
\label{thm:all-connection-occurrence}
Let $p\geq5$ be prime and put
\(
N:=\left\lfloor(p/2)^{1/3}\right\rfloor.
\)
For every $1\leq n\leq N$, we have Heuristic M holds in the logarithmic sense, i.e.



\begin{equation}
\label{eq:all-connection-logarithmic}
\log\Prob(\mathsf A_{E,n})
=
\frac12\log\frac np+O(\log\log p).
\end{equation}
\end{theorem}

\begin{proof}
Averaging $r_E(n)$ over $\Sp$, we obtain
\begin{equation}
\label{eq:all-connection-first-moment}
\begin{aligned}
\mathbb E[r_E]
&=
\frac1{h_p}
\sum_{d^2\mid n}A_p^{\BH}(n/d^2)\\
&=
\frac{\sqrt{pn}}{\pi h_p}
\sum_{d^2\mid n}
\frac{\mathcal L_{p,n/d^2}}d\\
&=
\frac{\sqrt{pn}}{\pi h_p}\mathcal H_{p,n}.
\end{aligned}
\end{equation}

We next bound the contribution from
multiple connecting isogenies.
For all $r\in\Z_{\geq 0}$, 
\[
0\leq
r-\mathbf1_{\{r>0\}}
\leq\binom r2.
\]
It gives
\begin{equation}
\label{eq:all-connection-bonferroni}
0\leq
\mathbb E[r_E]
-
\Prob(\mathsf A_{E,n})
\leq
\mathbb E\left[\binom{r_E}{2}\right]=
\frac1{2h_p}
\sum_{\substack{d^2\mid n\\e^2\mid n}}
C_p\left(\frac n{d^2},\frac n{e^2}\right).
\end{equation}
By \Cref{lem:general-mixed-pair-bound},
\[C_p\left(\frac n{d^2},\frac n{e^2}\right)=
O\left(
\frac n{de}\,
s\left(
\gcd\left(\frac n{d^2},\frac n{e^2}\right)
\right)
\bigl(1+\log(2n)\bigr)^2
\right).
\]
Then we bound the sum on the right hand side: for $d^2,e^2\mid n$, we have
\(
s\left(
\gcd\left(\frac n{d^2},\frac n{e^2}\right)
\right)
=
\frac{s(n)}{\operatorname{lcm}(d,e)}\). 
The sum $\sum\limits_{\substack{e^2\mid n,d^2\mid n\\e\geq1,d\geq1}}\frac{s(n)}{\operatorname{lcm}(d,e)}$ is uniformly bounded, since
\[
\begin{aligned}
\sum_{\substack{e^2\mid n,d^2\mid n\\d,e\geq1}}
\frac1{de\operatorname{lcm}(d,e)}&\leq
\sum_{d,e\geq1}
\frac1{de\operatorname{lcm}(d,e)}\\
&=
\sum_{d,e\geq1}
\frac{\gcd(d,e)}{d^2e^2}\\
&=
\zeta(2)^2
\sum_{r\geq1}\frac{\varphi(r)}{r^4}\\
&=
\frac{\zeta(2)^2\zeta(3)}{\zeta(4)}
<\infty.
\end{aligned}
\]
Consequently,
\[
\sum_{\substack{d^2\mid n\\e^2\mid n}}
C_p\left(\frac n{d^2},\frac n{e^2}\right)=
O\left(
ns(n)\bigl(1+\log(2n)\bigr)^2
\right).\]
By $h_p=p/12+O(1)$,
\[
\mathbb E\left[\binom{r_E}{2}\right]
=
O\left(
\frac{ns(n)}p
\bigl(1+\log(2n)\bigr)^2
\right).
\]
Combining this with
\eqref{eq:all-connection-first-moment}
and \eqref{eq:all-connection-bonferroni}, as well as $h_p=p/12+O(1)$
proves 
\begin{equation}
\label{eq:all-connection-arithmetic}
\Prob(\mathsf A_{E,n})
=\frac{12}{\pi}\sqrt{\frac{n}{p}}\mathcal H_{p,n}+
O\left(
\frac{ns(n)}p
\bigl(1+\log(2n)\bigr)^2
\right),
\end{equation}
where the implied constant is absolute.

Consider \[\frac{\Prob(\mathsf{A}_{E,n})}{\frac{12}{\pi}\sqrt{\frac{n}{p}}\mathcal H_{p,n}}=1+O\left(\sqrt{\frac{n}{p}}\frac{s(n)(1+\log (2n))^2}{\mathcal{H}_{p,n}}\right).\]
For every fixed $\eta>0$, Siegel's theorem
\cite{Siegel1935}, together with the
induced-character formula and Mertens'
product estimate, gives
\[
\mathcal L_{p,n}
\geq c_\eta p^{-\eta}
\]
uniformly for $1\leq n\leq N$ for a constant $c_\eta>0$.
Since all terms in $\mathcal H_{p,n}$ are positive,
\[
\mathcal H_{p,n}\geq\mathcal L_{p,n}
\geq c_\eta p^{-\eta}.
\]
Since $s(n)\leq\sqrt n$ and $n\leq N\leq p^{\frac{1}{3}}$, 
\[
\frac{
\Prob(\mathsf A_{E,n})
}{
(12/\pi)\sqrt{n/p}\,\mathcal H_{p,n}
}
=
1+
O_\eta\left(
p^{-1/6+\eta}(\log p)^2
\right).
\]
Choosing any fixed $0<\eta<1/6$
proves uniformly for $1\leq n\leq N$:
 \[\frac{\Prob(\mathsf{A}_{E,n})}{\frac{12}{\pi}\sqrt{\frac{n}{p}}\mathcal \mathcal H_{p,n}}=1+o(1).\]
Finally, for every fixed $\eta>0$,
\[
c_\eta p^{-\eta}
\leq \mathcal H_{p,n}
\leq C(\log p)^2.
\]
Since $\eta>0$ is arbitrary, this implies
\[
\log \mathcal H_{p,n}=O(\log\log p)
\]
uniformly for $1\leq n\leq N$.

Taking logarithms,
we obtain
\[
\begin{aligned}
\log\Prob(\mathsf A_{E,n})
&=
\frac12\log\frac np
+\log \mathcal H_{p,n}+O(1)\\
&=
\frac12\log\frac np+O(\log\log p).
\end{aligned}
\]
This proves \eqref{eq:all-connection-logarithmic}.
\end{proof}

\paragraph{Remark} From the proof, we have:
\begin{equation}
\label{eq:all-connection-normalized}
\frac{\Prob(\mathsf A_{E,n})}
{\sqrt{12n/p}}
=
\frac{\sqrt{12}}{\pi}\mathcal H_{p,n}
+
O\left(p^{-1/6}(\log p)^2\right).
\end{equation}
Thus, to disprove the proposed uniform
fixed-constant asymptotic in this family,
it suffices to construct a sequence
$(p_j,n_j)$ for which
\[
\mathcal H_{p_j,n_j}
\not\longrightarrow\frac{\pi}{\sqrt{12}}.
\]
However, there are lots of sequence without this limit. So the constant term in Heuristic M does not always hold.

\end{document}